\documentclass[reqno]{amsart}
\usepackage{amscd}
\usepackage{amssymb}
\usepackage{amstext}
\usepackage{latexsym}
\usepackage{graphicx}
\usepackage{epstopdf}
\newcommand{\fourierabs}[1]{\lfloor #1 \rfloor}
\newcommand{\vangle}{\measuredangle}
\newcommand{\embeds}{\hookrightarrow} 

\newcommand{\stFT}{\,\, \widetilde{} \,\,}

\newcommand{\init}{\vert_{t = 0}}

\newcommand{\abs}[1]{\left\vert #1 \right\vert}
\newcommand{\fixedabs}[1]{\vert #1 \vert}
\newcommand{\bigabs}[1]{\bigl\vert #1 \bigr\vert}

\newcommand{\norm}[1]{\left\Vert #1 \right\Vert}

\newcommand{\bignorm}[1]{\bigl\Vert #1 \bigr\Vert}

\newcommand{\C}{\mathbb{C}}

\newcommand{\R}{\mathbb{R}}

\newcommand{\innerprod}[2]{\left\langle \, #1 , #2 \, \right\rangle}

\newcommand{\angles}[1]{\langle #1 \rangle}

 \DeclareMathOperator{\im}{Im}

\newtheorem{theorem}{Theorem}

\newtheorem{lemma}{Lemma}

\theoremstyle{definition}

\theoremstyle{remark}

\title[DKG in three space dimensions]{Low Regularity local well-posedness
for the 1+3 dimensional Dirac-Klein-Gordon system}

\author[Achenef Tesfahun]{Achenef Tesfahun}

\thanks{Supported by the Research Council of Norway, project no.\
160192/V30, PDE and Harmonic Analysis. The author would like to
thank Sigmund Selberg for continuous support, encouragement and
advice while writing this paper.}


\begin{document}

\maketitle

\begin{abstract}
We prove that the Cauchy problem for the Dirac-Klein-Gordon system
of equations in 1+3 dimensions is locally well-posed in a range of
Sobolev spaces for the Dirac spinor and the meson field. The result
contains and extends the earlier known results for the same problem.
Our proof relies on the null structure in the system, and bilinear
spacetime estimates of Klainerman-Machedon type.
\end{abstract}

\section{Introduction} We consider the Dirac-Klein-Gordon system (DKG) in
three space dimensions,
\begin{equation}\label{DKG1}
\left\{
\begin{aligned}
   \bigl( D_t + \alpha\cdot D_x \bigr) \psi = \phi \beta \psi,
  \quad \quad &( D_t=-i\partial_t, \ D_x=-i\nabla )
  \\
  \square \phi = - \innerprod{\beta \psi}{\psi}, \quad \quad & (  \square = -\partial_t^2 + \Delta)
\end{aligned}
\right.
\end{equation}
with initial data
\begin{equation}\label{data}
   \psi \init = \psi_0 \in H^s, \qquad \phi \init =
\phi_0 \in H^r, \qquad
\partial_t \phi \init = \phi_1 \in H^{r-1},
\end{equation}
where $\psi(t,x)$ is the Dirac spinor, regarded as a column vector
in $\C^4$, and $\phi(t,x)$ is the meson field which is real-valued;
both the Dirac spinor and  the meson field are defined for $t\in \R,
\ x\in \R^{3}$; $M,m \ge 0$ are constants; $\nabla=(\partial_{x_1},
\partial_{x_2},
\partial_{x_3})$; $ \innerprod{u}{v} := \innerprod{u}{v}_{\C^4}=v^\dagger u$ for column
vectors $u, v \in \C^4$, where $v^\dagger$ is the complex conjugate
transpose of $v$;  $ H^s=(1+\sqrt{-\Delta})^{-s}L^2(\R^3)$ is the
standard Sobolev space of order $s$. The Dirac matrices are given in
$2 \times 2$ block form by
$$
   \beta =    \begin{pmatrix}
     I & 0  \\
     0 & -I
   \end{pmatrix},
   \qquad
   \alpha^j =  \begin{pmatrix}
     0 & \sigma^j  \\
     \sigma^j & 0
   \end{pmatrix},
$$
where
$$
   \sigma^1 =    \begin{pmatrix}
     0 & 1  \\
     1 &0
   \end{pmatrix},
   \qquad
   \sigma^2 =  \begin{pmatrix}
     0 & -i  \\
     i & 0
   \end{pmatrix},
   \qquad
   \sigma^3 =    \begin{pmatrix}
     1 & 0  \\
     0 & -1
   \end{pmatrix}
$$
are the Pauli matrices. The Dirac matrices $\alpha^j, \beta$ satisfy
\begin{equation}
  \label{DiracIdentity1}
  \beta^\dagger=\beta, \quad (\alpha^j)^\dagger = \alpha^j,
  \quad \beta^2 = (\alpha^j)^2 = I, \quad \alpha^j \beta + \beta \alpha^j
= 0.
  \end{equation}

 For the DKG system there are many conserved quantities which are not positive definite,
such as the energy, see ~\cite{gs79}.  However, there is a known
positive conserved quantity, namely the charge,
$\norm{\psi(t,.)}_{L^{2}}=\text{const}$. To study questions of
global regularity, a natural strategy is to study local (in time)
well-posedness (LWP) for low regularity data, and then try to
exploit the conserved quantities of the system. See, e.g., the
global result of Chadam ~\cite{c73} for 1+1 dimensional DKG system.
The LWP results for DKG in 1+3 dimensions are summarized in Table
\ref{Table1}\\

 For DKG in 1+3 dimensions the scale invariant
data is (see ~\cite{dfs2005})
 $$
  (\psi_0,\phi_0,\phi_1) \in L^2 \times \dot H^{1/2} \times \dot
  H^{-1/2},
$$  where $\dot
H^s=(\sqrt{-\Delta})^{-s}L^2$. Heuristically, one cannot expect
well-posedness below this regularity. This scaling also suggests
that $r=1/2+s$ is the line where equation \eqref{DKG1} is LWP.
Concerning LWP of the DKG system in 1+3 dimensions, the best result
to date is due to P. d'Ancona, D. Foschi and S. Selberg in
~\cite{dfs2005} for data $$ \psi_0 \in H^{\varepsilon}, \quad
\phi_0\in H^{1/2+\varepsilon}, \quad  \phi_1 \in
H^{-1/2+\varepsilon},$$ where $\varepsilon>0$ is arbitrary. This
result is arbitrarily close to the minimal regularity predicted by
the scaling ($\varepsilon=0$). The key achievement in this result is
that a null structure occurs not only in the Klein-Gordon part (in
the nonlinearity $ \innerprod{\beta \psi}{\psi} $) which was known
to be a null form (see ~\cite{dfs2005} for references)), but also in
the Dirac part (in the nonlinearity $\phi \beta \psi$) of the
system, which they discover using a duality argument. This requires
first to diagonalize the system by using the eigenspace projections
of the Dirac operator. The same authors used their result on the
null structure in $\phi \beta\psi$ to prove LWP below the charge
norm of the
DKG system in 1+2 dimensions (see ~\cite{dfs2006}). \\

In the present paper we study the LWP of the DKG system in 1+3
dimensions. We prove that \eqref{DKG1}--\eqref{data} is LWP for
$(s,r)$ in the convex region shown in Figure \ref{fig:1}, extending
to the right, which contains the union of all the results shown in
Table \ref{Table1} as a proper subset. In our proof, we take
advantage of the null structure in the nonlinearity $\phi\beta \psi$
found in ~\cite{dfs2005} besides the null structure in the
nonlinearity $\innerprod{\beta \psi}{\psi}$, and some bilinear spacetime estimates. \\

We now describe our main result.
\begin{theorem}\label{Mainthm}
Suppose $(s,r)\in \R^2$ belongs to the convex region described by
(see Figure \ref{fig:1}) the region
\begin{equation*}
  s>0,  \qquad
  \max\left(\frac{1}{2}+\frac{s}{3}, \frac{1}{3}+\frac{2s}{3},
  s\right)<r< \min \left(\frac{1}{2}+2s, 1+s\right).
\end{equation*}
Then the DKG system \eqref{DKG1} is LWP for data \eqref{data}.
Moreover, we can allow $r=1+s$ if $s>1/2$, and $r=s$ if $s>1$.
\end{theorem}

\begin{figure}[h]
   \centering
   \includegraphics{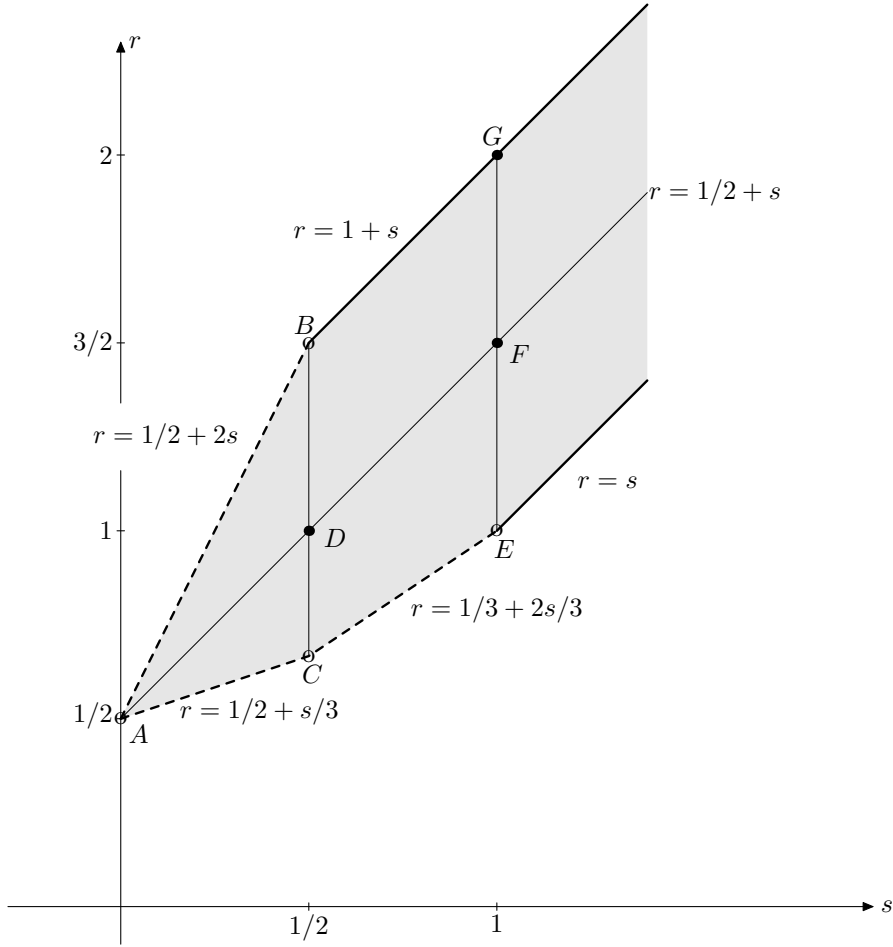}
   \caption{LWP holds in the interior of the shaded
region, extending to the right. Moreover, we can allow the line
$r=1+s$ for $s>1/2$, and the line $r=s$ for $s>1$. The line
$r=1/2+s$ represents the regularity predicted by the scaling.}
   \label{fig:1}
\end{figure}

 If $A, B, C, D$ are points in the $(s,r)$--plane, the symbol
 $AB$ represents a line from $A$ to $B$, $ABC$ represents a triangle and
$ABCD$ a quadrilateral,  all of them excluding the boundaries. We
use the following notation for different regions in Figure
\ref{fig:1}:
\begin{equation}\label{regions}
\begin{aligned}
&R_1:=ACD\cup AD,\\
& R_2:=ABD,\\
 &R_3:=D\cup F \cup CD \cup DF
\cup
FE \cup CDFE,\\
  &R_4:=G \cup BG\cup GF \cup BDGF,\\
& R:=BD \cup \bigcup_{j=1}^4 R_j .
\end{aligned}
\end{equation}

\begin{table}
\caption{LWP exponents for \eqref{DKG1}, \eqref{data}. That is, if
the data $(\psi_0, \phi_0, \phi_1)\in H^s \times H^r \times
H^{r-1}$, then there exists a time $T>0$ and a solution of
\eqref{DKG1}, $\bigl( \psi(t),\phi(t)\bigr) \in
C\bigl([0,T],H^s\bigr )\times C\bigl([0,T], H^r\bigr)$ which depends
continuously on the data. The solution is also unique in some
subspace of $C\bigl([0,T],H^s\bigl )\times C\bigl([0,T], H^r\bigr)$.
Here $\varepsilon>0$ is an arbitrary parameter.} \label{Table1}
\def\arraystretch{1.3}
\begin{center}
\begin{tabular}{|c|c|c|c|c|c|}
   \hline
     & $s$ & $r$
   \\
   \hline
   \hline
   classical methods & $1+\varepsilon$ & $3/2+\varepsilon$
   \\
   \hline
   Bachelot ~\cite{ab84}, 1984 & $1$ & $3/2$
   \\
   \hline
Strichartz estimate ~\cite{nb99, pt93},  1993 & $1/2+\varepsilon$ &
$1+\varepsilon$
   \\
   \hline
   Beals and Bezard ~\cite{bb96}, 1996 & $1$ & $2$
   \\
   \hline
Bournaveas ~\cite{nb99}, 1999 & $1/2$ & $1$
      \\
\hline
   Fang and Grillakis ~\cite{fg2005}, 2005 & $(1/4,1/2]$ & $1$
   \\
   \hline
D'Ancona, Foschi and Selberg ~\cite{dfs2005}, 2005 & $\varepsilon$ &
$1/2+\varepsilon$
\\
 \hline
\end{tabular}
\end{center}
\end{table}

\medskip\noindent

This paper is organized as follows. In the next section we fix some
notation, state definitions and basic estimates. In addition, we
shall rewrite the system \eqref{DKG1} by splitting $\psi$ as the sum
$P_+(D_x)\psi+ P_-(D_x)\psi$, where $P_\pm(D_x)$ are the projections
onto the eigenspaces of the matrix $\alpha.D_x$. We also state the
reduction of Theorem \ref{Mainthm} to two $X^{s,b}$ bilinear
estimates. In Section 3 we review the crucial null structure of the
bilinear forms involved, and we discuss product estimates for
wave-Sobolev spaces $H^{s, b}$. In Section 4 we interpolate between
the product estimates from Section 3 to get a wider range of
estimates. In Sections 5 and 6 we apply the estimates from Sections
3 and 4 to prove the  bilinear estimates from Section 2. In Section
7 we prove that these bilinear estimates are optimal
up to some endpoint cases, by constructing counterexamples.\\

For simplicity we set $ M=m=0$ in the rest of the paper, but the
discussion can easily be modified to handle the massive case as
well.\\

\section{Notation and preliminaries}\label{section 2}
In estimates, we use the symbols $\lesssim $,  $\simeq$,  $\gtrsim$
to denote relations $\le $, $=$, $\ge $  up to a positive constant
which may depend on $s$ and $r$. Also, if  $K_1 \lesssim K_2
\lesssim K_1$ we will write   $K_1 \approx K_2$. If in the
inequality $\lesssim $  the multiplicative constant is much smaller
than 1 then we use the symbol $\ll$; similarly, if in $ \gtrsim$ the
constant is much greater than 1 then we use $\gg$. Throughout we use
the notation $\angles{\cdot} = 1 + \abs{\cdot}$. The characteristic
function of a set $A$ is denoted by $\bold{1}_A$. For $a \in \R$,
$a^{\pm}:= a \pm \epsilon$ for sufficiently small $\epsilon>0$. The
Fourier transforms in space and space-time are defined by
$$
  \widehat f(\xi) = \int_{\R^3} e^{-ix\cdot\xi} f(x) \, dx,
  \qquad
  \widetilde u(\tau,\xi) = \int_{\R^{1+3}} e^{-i(t\tau+x\cdot\xi)} u(t,x)
\, dt \, dx.
$$
 Then $ \widetilde {D_t u}=\tau
\widetilde u$, and $\widetilde{ D_x u}= \xi \widetilde u$. If
$\phi:\R^3\rightarrow \C$, we define the multiplier $\phi(D)$ by
$$\widehat{\phi(D) f}(\xi) = \phi(\xi) \widehat f(\xi).$$ Given $u(t,x)$,
we denote by $\fourierabs{u}$ the function whose space-time Fourier
transform is $\abs{\widetilde u}$. If $X,Y,Z$ are normed function
spaces, we use the notation
  $ X \cdot Y \hookrightarrow Z$
to mean that $$\norm{uv}_Z \lesssim \norm{u}_X \norm{v}_Y.$$

In the study of non-linear wave equations it is standard that the
following spaces of Bourgain-Klainerman-Machedon
 type are used.
For $a, b \in \R$, define $X^{a,b}_\pm$, $H^{a,b}$ to be the
completions of $\mathcal S(\R^{1+3})$ with respect to the norms
\begin{align*}
   \norm{u}_{X^{a,b}_\pm} &= \bignorm{\angles{\xi}^a
\angles{\tau\pm \abs{\xi}}^b \widetilde
u(\tau,\xi)}_{L^2_{\tau,\xi}},
   \\
   \norm{u}_{H^{a,b}} &= \bignorm{\angles{\xi}^a \angles{\abs
{\tau} - \abs{\xi}}^b \widetilde u(\tau,\xi)}_{L^2_{\tau,\xi}},
          \end{align*}
    We also need the
restrictions to a time slab $
  S_T = (0,T) \times \R^3,
$ since we study local in time solutions.  The restriction
$X_\pm^{a,b}(S_T)$ is a Banach space with norm
$$ \norm{u}_{X_\pm^{a,b}(S_T)} = \inf_{ \tilde{u}_{\mid {S_T}}=u }
\norm{\tilde u}_{X_\pm^{a,b}}.$$  The restrictions $H^{a,b}(S_T)$ is
defined in the same way. We now collect some facts about these
spaces which will be needed in the later sections. It is well known
that the following interpolation property holds:
\begin{equation}\label{b-int}
\left(H^{s_0,\alpha_0},H^{s_1,\alpha_1}\right)_{[\theta]}=H^{s,\alpha},
  \end{equation} where $0\le\theta\le1$,  $s=(1-\theta)s_0+\theta s_1$,
$\alpha=(1-\theta)\alpha_0+\theta \alpha_1$  and  $(.,.)_{[\theta]}$
is the intermediate space with respect to the interpolation pair
(.,.). It immediately follows from a general bilinear complex
interpolation for Banach spaces (see for example ~\cite{bl76}) that
if
     \begin{align*}
   H^{a_0,\alpha_0 }\cdot H^{b_0,\beta_0}& \embeds H^{-c_0, -\gamma_0}, \\
    H^{a_1, \alpha_1} \cdot H^{b_1,\beta_1}&\embeds H^{-c_1,-\gamma_1},
  \end{align*}
  then
    \begin{align*}
    H^{a, \alpha}\cdot H^{b,\beta}\embeds H^{-c,-\gamma},
  \end{align*} where $0\le\theta\le1$, $a=(1-\theta)a_0+\theta a_1$,
$b=(1-\theta)b_0+\theta b_1$, $c=(1-\theta)c_0+\theta c_1$,
$\alpha=(1-\theta)\alpha_0+\theta \alpha_1$,
$\beta=(1-\theta)\beta_0+\theta \beta_1$ and
$\gamma=(1-\theta)\gamma_0+\theta \gamma_1$.\\

  We shall also need the fact that
\begin{align}\label{embxhc}
 X_\pm^{a,b} (S_T)\embeds H^{a,b} (S_T) &\embeds C\bigl([0,T], H^a\bigr)
\quad \text{provided} \ \ b>1/2,\\
  \label{embxh}
  X_\pm^{a,b} \embeds H^{a,b} &\quad \text{for all } \ \ b\ge 0.
\end{align} The embedding \eqref{embxhc} is equivalent to the estimate
 $$\norm{u(t)}_{H^{a}}\le C_1\norm{u}_{H^{a,b} (S_T)} \le C_2 \norm{u}_{
X_\pm^{a,b} (S_T)},$$
 for all $ 0\le t \le T$ and $C_1, C_2 \ge 1$.
 In the first inequality, $C_1$ will depend on $b$ (see ~\cite{dfs2005} for the proof), and the second inequality follows from the fact that
 $ \angles{\abs{\tau}-\abs{\xi}}\le \angles{\tau\pm\abs{\xi}}$ (hence
$C_2=1$), which also implies \eqref{embxh}. \\

Following ~\cite{dfs2005}, we diagonalize the system by defining the
projections
$$ P_{\pm}(\xi) = \frac{1}{2} \left( I \pm \hat\xi \cdot \alpha \right), $$
where $ \hat\xi \equiv \frac{\xi}{\abs{\xi}}$. Then the spinor field
splits into $\psi = \psi_+  +  \psi_-$, where $\psi_\pm = P_\pm(D_x)
\psi$. Now applying  $P_\pm(D_x)$ to the Dirac equation in
\eqref{DKG1}, and using the identities
\begin{equation}\label{identities}
\begin{aligned}
\alpha \cdot D_x &= \abs{D_x} P_+(D_x) - \abs{D_x}P_-(D_x),\\ \quad
P^2_\pm(D_x)= & P_\pm(D_x)  \quad \text{and} \quad
P_\pm(D_x)P_\mp(D_x)=0,
\end{aligned}
\end{equation}
 we obtain
\begin{equation}\label{DKG3}
\left\{
\begin{aligned}
  & \bigl( D_t + \abs{D_x} \bigr) \psi_+ =  P_+(D_x)(\phi \beta \psi),
  \\
  & \bigl( D_t - \abs{D_x} \bigr) \psi_- = P_-(D_x)(\phi \beta \psi),
  \\
  &\square \phi = - \innerprod{\beta \psi}{\psi},
\end{aligned}
\right.
\end{equation} which is the system we shall work with.

We iterate in the spaces
$$
   \psi_+ \in X_+^{s,\sigma}(S_T),
   \quad
   \psi_- \in X_-^{s,\sigma}(S_T),
   \quad
   (\phi,\partial_t \phi) \in H^{r,\rho} \times H^{r-1,\rho}(S_T),
$$
where
$$
   \frac{1}{2} < \sigma, \rho < 1
$$
will be chosen depending on $r,s$. By a standard argument (see
~\cite{dfs2005} for details) Theorem \ref{Mainthm} then reduces to
\begin{align}
\label{Bilinear-Dirac} \norm{ P_{\pm}(D_x)( \phi \beta
P_{[\pm]}(D_x)\psi )}_{ X^{s,\sigma-1+\varepsilon}_\pm }& \lesssim
\norm{ \phi }_{ H^{r,\rho} } \norm{\psi}_{ X^{s,\sigma}_{[\pm] } },\\
\label{Bilinear-KG}
   \norm{ \innerprod { \beta P_{[\pm]}(D_x)\psi } { P_{\pm}(D_x)\psi'}}_{
H^{r-1,\rho-1+\varepsilon} } &\lesssim
\norm{\psi}_{X^{s,\sigma}_{[\pm]} }\norm{\psi'}_{ X^{s,\sigma}_\pm
},
\end{align}
  for all $\phi, \psi, \psi' \in \mathcal{S}( \R^{1+3})$, where $\pm$ and $[\pm]$ denote
independent signs, and $\varepsilon>0$ is sufficiently small.

 But in ~\cite{dfs2005}, it was shown that
\eqref{Bilinear-Dirac} is equivalent, by duality, to an estimate
similar to \eqref{Bilinear-KG}, namely
\begin{equation}\label{Bilinear-DiracD}\tag{$\text{\ref{Bilinear-Dirac}}'$}
  \norm{\innerprod{\beta P_{[\pm]}(D_x)
\psi}{P_{\pm}(D_x)\psi'}}_{H^{-r,-\rho}}
  \lesssim \norm{\psi}_{{X_{[\pm]}^{s,\sigma}} }
  \norm{\psi'}_{X_{\pm}^{-s,1-\sigma-\varepsilon} },
\end{equation} for all $ \psi, \psi' \in \R^{1+3}$. Note that in
this formulation, the bilinear null form $\innerprod{\beta
P_{[\pm]}(D_x) \psi}{P_{\pm}(D_x)\psi'}$, \ appears again. Thus,
Theorem \ref{Mainthm} has been reduced to proving
\eqref{Bilinear-DiracD} and \eqref{Bilinear-KG}. We shall prove the
following theorem, which implies Theorem \ref{Mainthm}.
\begin{theorem}\label{Reducedthm} Suppose
\begin{equation}\label{condrs}
  s>0,  \qquad
  \max\left(\frac{1}{2}+\frac{s}{3}, \frac{1}{3}+\frac{2s}{3},
  s\right)<r< \min \left(\frac{1}{2}+2s, 1+s\right).
\end{equation}
Then there exist $1/2<\rho, \sigma <1$ and $\varepsilon>0$ such that
\eqref{Bilinear-DiracD} and \eqref{Bilinear-KG} hold simultaneously
for all $\psi,\psi' \in \mathcal S(\R^{1+3})$. Moreover, in addition
to \eqref{condrs} we can allow $r=1+s$ if $s>1/2$, and $r=s$ if
$s>1$. The parameters $\rho, \sigma$ can be chosen as follows:
\begin{equation}\label{condrho}
\rho=1/2+\varepsilon,
\end{equation}
\begin{equation}
 \sigma=
 \begin{cases}\label{condsigma}
 1/2+s/3   \quad &\text{if} \ \ (s,r) \in R_1, \\
1/2 + s  \quad  &\text{if} \  \ (s,r) \in R_2, \\
5/6-s/3+\varepsilon \quad   &\text{if} \ \ (s,r) \in R_3,\\
3/2-s+4\varepsilon \quad   &\text{if} \ \ (s,r) \in R_4  , \\
1-\varepsilon \quad    &\text{if} \ \ (s,r) \in BD ,  \\
\text{any number in }(1/2,1)  \quad &\text{otherwise},
 \end{cases}
\end{equation}
with $\varepsilon>0$ sufficiently small depending on $s,r$ (see
 \eqref{regions} to locate $(s,r)$ in the case of
\eqref{condsigma}).
\end{theorem}

\section{Null structure and a product law for wave Sobolev
spaces}\label{sectionnullprodlaw}

Let us first discuss the null structure in $\innerprod{\beta
P_{[\pm]}(D_x) \psi}{P_{\pm}(D_x)\psi'}$. The discussion here
follows ~\cite{dfs2005}. Taking the spacetime Fourier transform on
this bilinear form we get
\begin{multline*}
\left[\innerprod{\beta P_{[\pm]}(D_x)
\psi}{P_{\pm}(D_x)\psi'}\right]\stFT(\tau,\xi)\\
=\int_{\R^{1+3}} \innerprod{\beta P_{[\pm]}(\eta)\widetilde
\psi(\lambda,\eta)}
  {P_{\pm}(\eta-\xi)\widetilde \psi'(\lambda-\tau,\eta-\xi)} \, d\lambda
\, d\eta,
\end{multline*}
where we have $(\lambda-\tau, \eta-\xi)$ as an argument of $\tilde
\psi' $ instead of $(\tau-\lambda, \xi-\eta)$ because of the complex
conjugation in the inner product. Since
$P_{\pm}(\eta-\xi)^\dagger=P_{\pm}(\eta-\xi)$, and
$P_{\pm}(\eta-\xi)\beta=\beta P_{\mp}(\eta-\xi)$, we obtain
\begin{align*}
&\innerprod{\beta P_{[\pm]}(\eta)\widetilde \psi(\lambda,\eta)}
  {P_{\pm}(\eta-\xi)\widetilde
\psi'(\lambda-\tau,\eta-\xi)}\\
&=\innerprod{P_{\pm}(\eta-\xi)\beta P_{[\pm]}(\eta)\widetilde
\psi(\lambda,\eta)}
  {\widetilde \psi'(\lambda-\tau,\eta-\xi)}\\
  &=\innerprod{\beta P_{\mp}(\eta-\xi) P_{[\pm]}(\eta)\widetilde
\psi(\lambda,\eta)}
  {\widetilde \psi'(\lambda-\tau,\eta-\xi)}.
\end{align*}
The matrix $\beta P_{\mp}(\eta-\xi) P_{[\pm]}(\eta)$ is the symbol
of the bilinear operator $(\psi, \psi')\mapsto\innerprod{\beta
P_{[\pm]}(D_x) \psi}{P_{\pm}(D_x)\psi'}$.
 By
orthogonality, $P_{\mp}(\eta-\xi) P_{[\pm]}(\eta)$ vanishes when the
vectors $[\pm]\eta$ and $\pm(\eta-\xi)$ line up in the same
direction. The following lemma, proved in  ~\cite{dfs2005},
quantifies this cancellation. We shall use the notation
$\vangle(\eta,\zeta)$ for the angle between vectors $\eta,\zeta \in
\R^3$.

\begin{lemma}\label{NullLemma}
$\beta P_{\mp}(\eta-\xi) P_{[\pm]}(\eta) = O \left(
\vangle([\pm]\eta,\pm(\eta-\xi)) \right)$.
\end{lemma}

As a result of this lemma, we get
\begin{align}\label{NullFormEstimate}
 \abs{\innerprod{\beta P_{[\pm]}(D_x) \psi}{P_{\pm}(D_x)
\psi'}\stFT(\tau,\xi)}
  \lesssim \int_{\R^{1+3}} \theta_{[\pm],\pm} \abs{\widetilde
\psi(\lambda,\eta)}
  \abs{\widetilde \psi'(\lambda-\tau,\eta-\xi)} \, d\lambda \, d\eta,
\end{align}
where $\theta_{[\pm],\pm} =
\vangle\bigl([\pm]\eta,\pm(\eta-\xi)\bigr)$.\\

 The strategy for proving Theorem \ref{Reducedthm} is to make use of
this null form estimate, \eqref{NullFormEstimate}, and reduce
\eqref{Bilinear-DiracD} and \eqref{Bilinear-KG} to some well-known
bilinear spacetime estimates of Klainerman-Machedon type for
products of free waves. We now discuss some product laws for the
wave
Sobolev spaces $H^{a,\alpha}$ in the following theorems.\\

\begin{theorem}\label{prodthm1-spec}
Let $d>1/2$. Then
\begin{equation}\label{prodembed1spec}
   H^{a, d} \cdot H^{b, d} \embeds L^2,
\end{equation} provided that
\begin{align*}
&a, b\ge 0, \quad \text{and}\\
 & a+b>1.
\end{align*}
\end{theorem}
\begin{proof}
By the same proof as in Corollary 3.3 in ~\cite{bms2005}, but using
the dyadic estimates in Theorem 12.1 in ~\cite{fs2000}, we have, for
any $\varepsilon>0$,
\begin{equation*}
\norm{uv}_{L^2(\R^{1+3)}} \lesssim \norm{u_0}_{
H^{1+\varepsilon}(\R^3)} \norm{v_0}_{L^2(\R^3)}.
\end{equation*}
 It follows by the transfer principle  (see
~\cite{dfs2005}, Lemma 4) that
$$H^{1+\varepsilon, d} \cdot H^{0, d} \embeds L^2. $$ Now,
interpolation between
\begin{align*}
& H^{1+\varepsilon, d} \cdot H^{0, d} \embeds L^2, \\
 &H^{0, d} \cdot H^{1+\varepsilon , d} \embeds L^2,
\end{align*}
gives
$$H^{(1+\varepsilon)(1-\theta), d} \cdot H^{(1+\varepsilon)\theta, d} \embeds L^2,$$
for $\theta\in [0,1]$. If there exists $\theta\in [0,1]$ such that
$a\ge(1+\varepsilon)(1-\theta)$ ( $\Leftrightarrow \theta\ge
1-a/(1+\varepsilon))$ and $b\ge(1+\varepsilon)\theta$
($\Leftrightarrow \theta\le b/(1+\varepsilon)$), then we have
$$H^{a, d} \cdot H^{b, d} \embeds L^2.$$ If $a,b\ge 0$ and $a+b>1$, then such $\theta\in
[0,1]$ exists, if we choose $\varepsilon>0$ small enough. This
proves Theorem \ref{prodthm1-spec}.
\end{proof}

\begin{theorem}\label{bilhomthm}~\cite{fs2000, km93, km96}.
 Let $s_1, s_2, s_3 \in \R$.
For free waves $u(t)=e^{\pm it\abs{D_x}}u_0$ and $v(t)=e^{[\pm]
it\abs{D_x}}u_0$ (where $\pm$ and $[\pm]$ are independent signs), we
have the estimate
\begin{equation}
\norm{\abs{D_x}^{-s_3}(uv)}_{L^2(\R^{1+3)}}\lesssim \norm{u_0}_{\dot
H^{s_1}} \norm{v_0}_{\dot H^{s_2}}
\end{equation}
if and only if
\begin{equation}\label{prodconda}
s_1 + s_2 + s_3=1, \quad s_1 + s_2 >1/2, \quad
       s_1,s_2 <1.\\
\end{equation}
\end{theorem}

 As an application of Theorems
\ref{prodthm1-spec} and \ref{bilhomthm} we have the following:

\begin{theorem}\label{prodthm1}
Suppose $s_1, s_2, s_3 \in \R$ \ and \ $d>1/2$. Then
\begin{equation}\label{prodembed1}
   H^{s_1, d} \cdot H^{s_2, d} \embeds H^{-s_3,0}
\end{equation}
provided $s_1, s_2, s_3$ satisfy
\begin{equation}\label{prodcond1}
\begin{aligned}
       &s_1 + s_2 + s_3=1, \quad s_1 + s_2 >1/2, \\
       &s_1 + s_3 \ge 0, \quad s_2 + s_3 \ge 0 ,\\
       &s_1, s_2 <1,
\end{aligned}
\end{equation}
or
   \begin{equation}\label{prodcond11}
\begin{aligned}
       &s_1 + s_2 + s_3>1, \quad s_1 + s_2 >1/2, \\
       &s_1 + s_3 \ge 0, \quad s_2 + s_3 \ge 0 .\\
\end{aligned}
\end{equation}
\end{theorem}
\begin{proof} First, let us prove \eqref{prodembed1} for $s_1,s_2,s_3 \in \R$ satisfying \eqref{prodcond1}.
 By Theorem \ref{bilhomthm} and the transfer principle (see
~\cite{dfs2005}, Lemma 4), we obtain
\begin{equation}\label{prodembeda}
   H^{s_1, d} \cdot H^{s_2, d} \embeds H^{-s_3,0} \quad  \text{if} \ \
    \begin{cases}
& s_1 + s_2 + s_3=1,\\
& s_1 + s_2 >1/2, \\
 &      s_1,s_2, s_3 \ge 0, \quad s_1, s_2<1.
 \end{cases}
\end{equation}
Note that in view of \eqref{prodcond1} at most one of $s_1, s_2,
s_3$ can be $\le 0$. But by the triangle inequality in Fourier space
(i.e., Leibniz rule), we can always reduce the problem to the case
$s_1, s_2,s_3 \ge 0$. Indeed, if $s_3\le 0$, then \eqref{prodembed1}
reduces to  $$ H^{s_1+s_3, b} \cdot H^{s_2, d} \embeds L^2 \quad
\text{and} \quad H^{s_1, d} \cdot H^{s_2+s_3, d} \embeds L^2.$$  In
view of \eqref{prodembeda} these estimates hold for $s_1,s_2,s_3$
satisfying \eqref{prodcond1}. If $s_1\le 0$, then \eqref{prodembed1}
reduces to
$$ H^{0, d} \cdot H^{s_1+s_2, d} \embeds H^{-s_3,0} \quad \text{
and} \quad H^{0,d} \cdot H^{s_2, d} \embeds H^{-(s_1+s_3), 0},$$ and
again by
 \eqref{prodembeda} these hold for $s_1,s_2,s_3$
satisfying \eqref{prodcond1}.
 The case $s_2\le 0$ is symmetrical to that of $s_1 \le 0$.
\\ \\
 It remains to show  \eqref{prodembed1} for $s_1,s_2,s_3$ satisfying
 \eqref{prodcond11}. Write $s_1+s_2+s_3=1+\varepsilon$ where $\varepsilon>0$.
 We consider three cases: $s_3\le 0$, $0<s_3<1/2$ \ and \ $s_3\ge
 1/2$.\\

 \paragraph{Case 1: $s_3\le 0$} In this case (using $s_3=1+\varepsilon-s_1-s_2$), \eqref{prodembed1} reduces
 to
 $$\
H^{1+\varepsilon-s_2, d} \cdot H^{s_2, d} \embeds L^2 \quad
\text{and}\quad
 H^{s_1, d} \cdot H^{1+\varepsilon-s_1, d} \embeds L^2,
$$
which hold by Theorem \ref{prodthm1-spec} (since $s_1,s_2\ge 0$, by \eqref{prodcond11} and the assumption $s_3\le 0$).\\

 \paragraph{Case 2: $0<s_3<1/2$} Here we consider three subcases:
 $s_1\le 0$, $s_2\le 0$ and $s_1, s_2 \ge 0 $. By symmetry it suffices to
 consider $s_1\le 0$ and $s_1, s_2 \ge 0 $.\\ \\
 Assume $s_1 \le 0$; then
 (using $s_3=1+\varepsilon-s_1-s_2$) \eqref{prodembed1} reduces
 to
 \begin{align}\label{pe-case21}
 &H^{0, d} \cdot H^{1+\varepsilon-s_3, d}  \embeds H^{-s_3,0}\\
 \label{pe-case22}
 &H^{0, d} \cdot H^{1+\varepsilon-s_1-s_3, d} \embeds H^{-(s_1+s_3),
 0}.
\end{align}
Since \eqref{prodcond1} implies \eqref{prodembed1}, we have
$$ H^{0, d}
\cdot H^{1/2+\varepsilon, d} \embeds H^{-(1/2-\varepsilon),0}\embeds
H^{-1/2,0}.$$  Interpolating between this and
$$
H^{0, d} \cdot H^{1+\varepsilon, d}  \embeds L^2,
$$
with $\theta=2s_3$, gives \eqref{pe-case21} (note that $\theta\in
(0,1)$ by the assumption on $s_3$).  The same interpolation, but now
with $\theta=2(s_1+s_3)$ ($\theta \in [0,1]$ by the assumption on
$s_1$ and $s_3$), gives \eqref{pe-case22}.

 Assume next  $s_1, s_2
\ge 0 $. Choose $0\le s_1'\le s_1$, $0\le s_2'\le s_2$ such that
$s_1',s_2'<1$ and $s_1'+s_2'+s_3=1$. Indeed, we can choose such
$s_1'$ and $s_2'$ as follows: If $s_2+s_3\le 1$, take
$s_1':=1-(s_2+s_3)\in [0,1)$ and $s_2':=s_2 \in [0,1)$. If
$s_2+s_3>1$, take $s_1':=0$ and $s_2':=1-s_3\in (1/2,1).$ Then the
problem reduces to
$$
 H^{s_1', d} \cdot H^{s_2', d}
\embeds H^{-s_3,0},
$$ which holds since \eqref{prodcond1} implies \eqref{prodembed1}.\\

\paragraph{Case 3: $s_3\ge1/2$} Take $s_3'=1/2-\delta$, where $\delta>0$ is chosen such that
$s_1+s_2+s_3'>1$ (this is possible due to the assumption
$s_1+s_2>1/2$ in \eqref{prodembed1}). Then
$$
 H^{-s_3',0}\embeds H^{-s_3,0},
$$
so the problem reduces to case 2 for $s_1$,  $s_2$ and $s_3'$.
\end{proof}

We also need the following product law for the Wave Sobolev spaces.
\begin{theorem}\label{prodthm2}~\cite{s99}.
Let $t_1,t_2,t_3 \in \R$.  Then
\begin{equation}\label{prodembed2}
   H^{t_1, d_1} \cdot H^{t_2, d_2} \embeds H^{-t_3,-d_3}
\end{equation}
provided
\begin{equation}\label{prodcond2}
\begin{aligned}
      & t_1+t_2+t_3 > 3/2,\\
 &  t_1 +t_2 \ge 0, \quad  t_2+t_3 \ge 0, \quad t_1+t_3\ge 0\\
 & d_1+d_2+d_3>1/2,\\
 & d_1,d_2,d_3\ge 0.
 \end{aligned}
\end{equation}
   Moreover,  we can allow
$t_1 + t_2 + t_3 = 3/2$, provided $t_j \neq 3/2$ for $1\le j \le 3$.
Similarly, we may take $d_1 + d_2 + d_3 = 1/2$, provided $d_j \neq
1/2$ for $1 \le j \le 3$.
    \end{theorem}

\begin{proof}
In view of \eqref{prodcond2}, at most one of $t_1, t_2, t_3$ can be
negative. But by the same Leibniz rule as in the proof of Theorem
\ref{prodthm1} this can be reduced to the case $t_1, t_2, t_3 \ge
0$, which was proved in \cite[Proposition ~10]{s99}.
\end{proof}

\begin{theorem}\label{speprodthm3} Let $\epsilon>0$. Then
\begin{equation}\label{spemb}
   H^{1/2+\epsilon, {1/2}^+ } \cdot H^{ \epsilon, {1/2}^+ } \embeds
H^{-1+\epsilon, 1/2} .
\end{equation}
    \end{theorem}
\begin{proof} The embedding \eqref{spemb} is equivalent to the estimate
$$
I\lesssim \norm{u}_{L^2(\R^{1+3})}\norm{u}_{L^2(\R^{1+3})},
$$
where
  $$
   I=\norm{  \int_{\R^{1+3}} \frac{ \angles{ \abs{\tau}-\abs{\xi} }^{1/2}
\widetilde u(\lambda,\eta) \widetilde v(\tau-\lambda, \xi-\eta)}
   { \angles{\xi}^{1-\epsilon}  \angles{\eta}^{1/2+\epsilon}  \angles{
\xi-\eta }^{\epsilon} \angles{ \abs{\lambda}-\abs{\eta} }^{{1/2
}^+}\angles{ \abs{ \tau-\lambda}-\abs{\xi-\eta} }^{{1/2}^+}}d\lambda
d\eta  }_{L^2_{(\tau,\xi)}}.$$
   By the 'hyperbolic' Leibniz rule (see ~\cite{ks2002} lemma 3.2), we reduce this
   to three estimates
\begin{align*}
& H^{1/2+\varepsilon, 0} \cdot H^{\varepsilon, {1/2}^+} \embeds H^{-1+\varepsilon, 0}, \\
& H^{1/2+\varepsilon, {1/2}^+} \cdot H^{\varepsilon, 0} \embeds
H^{-1+\varepsilon, 0},
\end{align*}
and (using also transfer principle to one free wave estimate)
\begin{equation*}
\norm{\abs{D_x}^{{-1+\epsilon}}D_-^{1/2}(u v)}_{L^2}\lesssim
\norm{\abs{D_x}^{1/2+\epsilon/2} u_0}_{  L^2}
\norm{\abs{D_x}^{\epsilon/2} v_0 }_{L^2},
 \end{equation*} where $u=e^{\pm
it\abs{D_x}}u_0$ and $v=e^{\pm it\abs{D_x}}v_0$, and the operator
$D_-$ corresponds to the symbol
 $\abs{\abs{\tau}-\abs{\xi}}$. The first two estimates hold by Theorem \ref{prodthm2}, and the last estimate holds by Theorem 1.1 in
 ~\cite{fs2000}.
\end{proof}

\section{Interpolation results}
By bilinear interpolation between special cases of Theorems
\ref{prodthm1} and \ref{prodthm2}, and at one point Theorem
\ref{speprodthm3}, we obtain a series of estimates which will be
useful in the proof of Theorem \ref{Reducedthm}. For $a,b,c, \alpha,
\beta, \gamma \in \R$, and $\epsilon>0$ sufficiently small, we
obtain the following estimates (the proof is given below):

\begin{equation}
\label{bie1} H^{a, \alpha } \cdot H^{0, {1/2}^+} \embeds H^{-c, 0}
\quad
\text{if} \quad  \begin{cases} & a, c, \alpha\ge 0,\\
 & 3\min(a/2,\alpha)+ c>3/2.\\
\end{cases}
\end{equation}

\begin{equation}
\label{bie1e} H^{a, \alpha } \cdot H^{0, {1/2}^+} \embeds H^{-c, 0}
\quad \text{if} \quad  \begin{cases}
&a, \alpha\ge 0, \ c\ge 1/2,\\
& \min(a,\alpha)+ c/2>3/4.\\
\end{cases}
\end{equation}

\begin{equation}\label{bie2}
H^{a, \alpha } \cdot H^{0, \beta} \embeds H^{0, -\gamma} \quad
\text{if} \quad  \begin{cases} \ \ &a>1, \ \alpha> 0, \ \beta, \gamma\ge 0, \\
& a+ \min(\alpha, \beta)>3/2,\\
& \gamma + \min(\alpha, \beta)>1/2.
\end{cases}
\end{equation}
\begin{equation}\label{bie3}
H^{a, {1/2}^+ } \cdot H^{b, \beta} \embeds H^{-c, 0} \quad
\text{if}\quad
\begin{cases} \ &c,\beta\ge 0, \ a,b>0,\\
 & a+b=1, \\
 &   c+\beta>1/2.
\end{cases}
\end{equation}
\begin{equation}\label{bie4}
H^{1, {1/2}^+ } \cdot H^{0, \beta} \embeds H^{-c, 0} \quad
\text{if}\quad
\begin{cases} &\beta\ge 0, \ c>0, \\
   &c+\beta>1/2.
\end{cases}
\end{equation}
\begin{equation}\label{bie5}
H^{a, \alpha } \cdot H^{b, {1/2}^+} \embeds H^{-c, 0} \quad
\text{if}\quad
\begin{cases} &a, b,\alpha \ge 0, \ c \ge 1/2, \\
&\min(a,\alpha)+ 2b/3>1/2,\\
&\min(a,\alpha)+ 2c>3/2.
\end{cases}
\end{equation}
\begin{equation}\label{bie6}
H^{a, {1/2}^+ } \cdot H^{b, \beta} \embeds L^2 \quad \text{if}\quad
 \begin{cases}  &  b, \beta \ge 0, \ a \ge 1/2, \\
&a+ 2\min(b,\beta)>3/2.
\end{cases}
\end{equation}
\begin{equation}\label{bie7}
H^{a, {1/2}^+ } \cdot H^{1/2, \beta} \embeds L^2 \quad
\text{if}\quad
\begin{cases} &  \beta \ge 0, \ a \ge 1/2, \\
&a+ \beta>1.
\end{cases}
\end{equation}
\begin{equation}\label{bie8}
H^{a, {1/2}^+ } \cdot H^{\epsilon, \beta} \embeds H^{-1+\epsilon,
-\gamma} \quad  \text{if}\quad \begin{cases}
&a,\beta\ge 0, \ \gamma\ge -1/2,\\
&\min(a,\beta)+\gamma/2>1/4.
\end{cases}
\end{equation}
\begin{equation}\label{bie9}
H^{1/2, {1/2}^+ } \cdot H^{0, \beta} \embeds H^{-c, 0} \quad
\text{if}\quad  \begin{cases}
 & \beta \ge 0, \ c > 1/2, \\
&c+\beta>1.
\end{cases}
\end{equation}
\begin{proof}[Proof of \eqref{bie1}--\eqref{bie9}]
The parameter $\varepsilon>0$ is assumed to be sufficiently small.

To prove \eqref{bie1} we interpolate between
\begin{align*}
H^{1+\varepsilon, 1/2+\varepsilon} \cdot H^{0, {1/2}^+} &\embeds L^2,\\
L^2 \cdot H^{0, {1/2}^+} & \embeds H^{-(3/2+\varepsilon), 0}.
\end{align*}
This gives
$$H^{(1+\varepsilon)(1-\theta),
(1/2+\varepsilon)(1-\theta)} \cdot H^{0, {1/2}^+} \embeds
H^{-(3/2+\varepsilon)\theta, 0}$$ for $\theta\in[0,1]$. Now, if
there exists $\theta\in[0,1]$ such that $a\ge
(1+\varepsilon)(1-\theta)$ \ $\left(\Leftrightarrow \theta \ge
1-a/(1+\varepsilon) \right)$, \ $\alpha \ge
(1/2+\varepsilon)(1-\theta)$ \ $\left(\Leftrightarrow \theta \ge
1-2\alpha/(1+2\varepsilon) \right)$ and $c\ge
(3/2+\varepsilon)\theta$ \ $\left(\Leftrightarrow \theta \le
2c/(3+2\varepsilon) \right)$, then we have $H^{a, \alpha } \cdot
H^{0, {1/2}^+} \embeds H^{-c, 0}.$ But such a $\theta\in[0,1]$
exists if $a,\alpha, c\ge 0$, $3a+2c\ge
3+5\varepsilon-2\varepsilon(a+c)+2\varepsilon^2$ and $2c+6\alpha\ge
3+8\varepsilon-2\varepsilon(c+\alpha)+4\varepsilon^2$. Since
$\varepsilon>0$ is very small, it is enough to have $a,\alpha,c\ge
0$, $3a+2c>3$ and $2c+6\alpha>3$. This proves \eqref{bie1}.
Interpolation between
\begin{align*}
H^{1/2+\varepsilon, 1/2+\varepsilon} \cdot H^{0, {1/2}^+} &\embeds H^{-(1/2+\varepsilon), 0},\\
L^2 \cdot H^{0, {1/2}^+} & \embeds H^{-(3/2+\varepsilon), 0},
\end{align*}
  with a similar
argument as above, proves  \eqref{bie1e}.

To prove \eqref{bie2}, we interpolate between
\begin{align*}
H^{1+\varepsilon, 1/2+\varepsilon} \cdot H^{0, 1/2+\varepsilon} &\embeds L^2,\\
H^{3/2+\varepsilon, \varepsilon} \cdot L^2 &\embeds H^{0,
-(1/2-\varepsilon)}.
\end{align*}
This gives
\begin{equation*}
H^{(1+\varepsilon)(1-\theta)+( 3/2+\varepsilon)\theta,
(1/2+\varepsilon)(1-\theta)+\varepsilon\theta } \cdot H^{0,(
1/2+\varepsilon)(1-\theta)} \embeds H^{0, -(1/2-\varepsilon)\theta},
\end{equation*}
for $\theta\in[0,1]$. If there exists $\theta\in[0,1]$ such that
$a\ge (1+\varepsilon)(1-\theta)+ ( 3/2+\varepsilon)\theta$, $\alpha
\ge (1/2+\varepsilon)(1-\theta)+ \varepsilon\theta$, $\beta\ge(
1/2+\varepsilon)(1-\theta) $ \ and $\gamma \ge
(1/2-\varepsilon)\theta$ , then we have
$$H^{a, \alpha }\cdot H^{0, \beta} \embeds H^{0, -\gamma}. $$
By a similar argument as in the proof of \eqref{bie1}, such a
$\theta\in[0,1]$ exists if $a>1$, $\alpha>0$, $\beta, \gamma\ge 0$,
$a+\alpha>3/2$, $a+\beta>3/2$, $\alpha +\gamma>1/2$ and
$\beta+\gamma>1/2$. This proves \eqref{bie2}.

To prove \eqref{bie3}, we interpolate between
\begin{align*}
H^{a, {1/2}^+} \cdot H^{b, 1/2+\varepsilon} &\embeds L^2,\\
H^{a, {1/2}^+} \cdot H^{b, 0} &\embeds H^{-1/2, 0},
\end{align*}
which both hold true if $a+b=1$, $a,b>0$, by Theorems \ref{prodthm1}
and \ref{prodthm2}, respectively. This gives
$$ H^{a, {1/2}^+} \cdot H^{b, ( 1/2+\varepsilon)(1-\theta)}\embeds H^{-\theta/2, 0}
$$ for $\theta\in[0,1]$. If there exists
$\theta\in[0,1]$ such that $\beta\ge ( 1/2+\varepsilon)(1-\theta)$
and $c\ge \theta/2$, then we have
$$H^{a, {1/2}^+} \cdot H^{b, \beta}\embeds H^{-c, 0}, $$ for $a+b= 1$, $a,b>0$.
By a similar argument as before such a $\theta\in[0,1]$ exists if
$\beta, c\ge 0$ and $c+\beta>1/2$.

For \eqref{bie4}--\eqref{bie9}, similar arguments as in the proof of
\eqref{bie1} are used, so we only give the interpolation pairs,
which give the desired estimate when interpolated.

For \eqref{bie4}, we use
\begin{align*}
H^{1, {1/2}^+} \cdot H^{0, 1/2+\varepsilon} &\embeds H^{-\varepsilon, 0},\\
H^{1, {1/2}^+} \cdot L^2 &\embeds H^{-1/2, 0}.
\end{align*}
 For \eqref{bie5},
we interpolate between
\begin{align*}
H^{1/2+\varepsilon, 1/2+\varepsilon} \cdot H^{0, {1/2}^+} &\embeds H^{-(1/2-\varepsilon),0},\\
L^2 \cdot H^{3/4, {1/2}^+} &\embeds H^{-3/4, 0}.
\end{align*}

 For \eqref{bie6},
we interpolate between
\begin{align*}
H^{1/2, {1/2}^+} \cdot H^{1/2, 1/2+\varepsilon} &\embeds L^2,\\
 H^{3/2 +\varepsilon, {1/2}^+} \cdot L^2 &\embeds L^2.
\end{align*}

For \eqref{bie7}, we interpolate between
\begin{align*}
H^{1/2, {1/2}^+} \cdot H^{1/2, 1/2+\varepsilon} &\embeds L^2,\\
 H^{1, {1/2}^+} \cdot H^{1/2,0} &\embeds L^2.
\end{align*}

For \eqref{bie8}, we interpolate between
\begin{align*}
 H^{0, {1/2}^+} \cdot H^{\varepsilon,0} &\embeds
 H^{-(1-\varepsilon),-(1/2+\varepsilon)},\\
 H^{1/2+\varepsilon, {1/2}^+} \cdot H^{\varepsilon, 1/2+\varepsilon} &\embeds H^{-(1-\varepsilon),
 1/2},
 \end{align*}
where the second embedding holds by Theorem \ref{speprodthm3}.

For \eqref{bie9}, we interpolate between
\begin{align*}
H^{1/2, {1/2}^+} \cdot H^{0, 1/2+\varepsilon} &\embeds H^{-(1/2+\varepsilon), 0},\\
 H^{1/2, {1/2}^+} \cdot L^2 &\embeds H^{-1,0},
\end{align*}
where the first embedding does not directly follow from Theorems
\ref{prodthm1} and \ref{prodthm2}, but from interpolation between
\begin{align*}
H^{1/2+\varepsilon, {1/2}^+} \cdot H^{0, 1/2+\varepsilon} &\embeds H^{-(1/2-\varepsilon), 0},\\
 H^{0, {1/2}^+} \cdot H^{0, 1/2+\varepsilon}&\embeds
 H^{-(3/2+\varepsilon),0},
\end{align*}
which gives
$$ H^{(1/2+\varepsilon)(1-\theta), {1/2}^+} \cdot H^{0, 1/2+\varepsilon} \embeds H^{-(1/2-\varepsilon)(1-\theta)-(3/2+\varepsilon)\theta, 0}
$$ for $\theta\in [0,1]$. Choosing
$\theta=\frac{2\varepsilon}{1+2\varepsilon}$ gives the desired
estimate.

\end{proof}

In the following two sections, we shall present the proof of the
bilinear estimates \eqref{Bilinear-DiracD} and \eqref{Bilinear-KG}
for all $\psi,\psi' \in \mathcal S(\R^{1+3})$ provided $(r,s)$,
$\rho$ and $\sigma$ are as in \eqref{condrs}, \eqref{condrho} and
\eqref{condsigma} respectively. These will imply Theorem
\ref{Reducedthm}. First we prove \eqref{Bilinear-KG}, and then
\eqref{Bilinear-DiracD}. Note that using \eqref{embxh} we can reduce
$X^{s,b}$ type estimates to $H^{s,b}$ type estimates, which we shall
do in the following two sections.

\section{Proof of \eqref{Bilinear-KG}}\label{proofBilinear-KG}
 Without loss of generality
we take $[\pm] = +$. Assume $\psi,\psi' \in \mathcal S(\R^{1+3})$ .
Using \eqref{NullFormEstimate}, we can reduce \eqref{Bilinear-KG}
(write $\rho=1/2+\varepsilon$, as in \eqref{condrho}) to
$$
  I^{\pm}
  \lesssim \norm{\psi}_{{X_+^{s,\sigma}}}
  \norm{\psi'}_{X_\pm^{s,\sigma}},
$$
where
$$
  I^{\pm} = \norm{\int_{\R^{1+3}} \frac{ \theta_\pm}{
\angles{\xi}^{1-r}\angles{\abs{\tau}-\abs{\xi}}^{1/2-2\varepsilon}}
\abs{\widetilde \psi(\lambda,\eta)}
  \abs{\widetilde \psi'(\lambda-\tau,\eta-\xi)} \, d\lambda \,
d\eta}_{L^2_{\tau,\xi}},
$$
  and $$
  \theta_\pm = \vangle\bigl(\eta,\pm(\eta-\xi)\bigr).$$\\
The low frequency case, where
$\min(\fixedabs{\eta},\fixedabs{\eta-\xi}) \le 1$ in $I^{\pm}$,
follows from a similar argument as in ~\cite{dfs2006}, and hence we
do not consider this question here. From now on we assume that in
$I^{\pm}$,
\begin{equation}\label{FrequencySupport}
  \fixedabs{\eta}, \fixedabs{\eta-\xi} \ge 1.
\end{equation}
We shall use the following notation in order to make expressions
manageable:
\begin{gather*}
  F(\lambda,\eta) = \angles{\eta}^s \angles{\lambda+\fixedabs{\eta}}^{\sigma}
  \abs{\widetilde \psi(\lambda,\eta)},
  \qquad
  G_\pm(\lambda,\eta) = \angles{\eta}^{s}
\angles{\lambda\pm\fixedabs{\eta}}^{\sigma}
  \abs{\widetilde \psi'(\lambda,\eta)},
  \\
  \Gamma = \abs{\tau}-\abs{\xi},
  \qquad \Theta = \lambda+\fixedabs{\eta},
  \qquad \Sigma_\pm = \lambda-\tau\pm\fixedabs{\eta-\xi},
  \\
  \kappa_{+} = \abs{\xi} - \bigabs{ \fixedabs{\eta} - \fixedabs{\eta-\xi}},
  \qquad
  \kappa_{-} = \fixedabs{\eta} + \fixedabs{\eta-\xi} - \abs{\xi}.
\end{gather*}
 We shall need the estimates (see ~\cite{dfs2005}):
\begin{equation}\label{ThetaEstimates}
  \theta_{+}^2 \sim \frac{\abs{\xi}
  \kappa_{+}}{\fixedabs{\eta}\fixedabs{\eta-\xi}},
  \qquad
  \theta_{-}^2 \sim \frac{(\fixedabs{\eta}+\fixedabs{\eta-\xi})
  \kappa_{-}}{\fixedabs{\eta}\fixedabs{\eta-\xi}}
  \sim \frac{\kappa_-}{\min(\fixedabs{\eta},\fixedabs{\eta-\xi})}.
\end{equation}
\begin{align}
  \label{rEstimateA}
  \kappa_{\pm} &\le 2 \min(\fixedabs{\eta},\fixedabs{\eta-\xi}),
  \\
  \label{rEstimateB}
  \kappa_{\pm} &\le \abs{\Gamma} + \abs{\Theta} + \abs{\Sigma_\pm}.
\end{align}

\subsection{Estimate for $I^+$} By \eqref{ThetaEstimates}, and
using \eqref{FrequencySupport}
$$
  I^+ \lesssim
  \norm{\int_{\R^{1+3}} \frac{\kappa_+^{1/2} F(\lambda,\eta)
G_+(\lambda-\tau,\eta-\xi)}
  { \angles{\xi}^{1/2-r}\angles{\eta}^{1/2+s} \angles{\eta-\xi}^{1/2+s}
\angles{\Gamma}^{1/2-2\varepsilon}
   \angles{\Theta}^{\sigma} \angles{\Sigma_+}^{\sigma}} \, d\lambda \,
d\eta}_{L^2_{\tau,\xi}}.
$$
By \eqref{rEstimateA} and \eqref{rEstimateB}
$$\kappa_+^{1/2} \lesssim \abs{\Gamma}^{1/2-2\varepsilon}
\min(\fixedabs{\eta},\fixedabs{\eta-\xi})^{2\varepsilon} +
\abs{\Theta}^{1/2} + \abs{\Sigma_+}^{1/2}.$$ Moreover, by symmetry
we may assume $\fixedabs{\eta}\ge\fixedabs{\eta-\xi}$ in $I^+$. By
\eqref{condrs}, $r> 1/2$, so we have by the triangle inequality
\begin{equation}\label{Lebz}
\angles{\xi}^{r-1/2}\lesssim
\angles{\eta}^{r-1/2}+\angles{\eta-\xi}^{r-1/2}\lesssim
\angles{\eta}^{r-1/2}.
\end{equation}
Hence the estimate reduces to
$$
  I_j^+ \lesssim \norm{F}_{L^2}
  \norm{G_+}_{L^2}, \ \ j=1,2, 3 ,
$$
where
\begin{align*}
  I^+_1 &= \norm{\int_{\R^{1+3} }
  \frac{ F(\lambda,\eta) G_+(\lambda-\tau,\eta-\xi)}
  {\angles{\eta}^{1+s-r}
\angles{\eta-\xi}^{1/2+s-2\varepsilon}
   \angles{\Theta}^{\sigma} \angles{\Sigma_+}^{\sigma}} \, d\lambda \,
d\eta}_{L^2_{\tau,\xi}},
\\
  I^+_2 &= \norm{\int_{\R^{1+3} }
  \frac{ F(\lambda,\eta) G_+(\lambda-\tau,\eta-\xi)}
  {  \angles{\eta}^{1+s-r} \angles{\eta-\xi}^{1/2+s}
  \angles{\Gamma}^{1/2-2\varepsilon} \angles{\Theta}^{\sigma-1/2}
\angles{\Sigma_+}^{\sigma} } \, d\lambda \, d\eta}_{L^2_{\tau,\xi}},
\\
  I^+_3 &= \norm{\int_{\R^{1+3} }
  \frac{ F(\lambda,\eta) G_+(\lambda-\tau,\eta-\xi)}
  {\angles{\eta}^{1+s-r}
\angles{\eta-\xi}^{1/2+s}
   \angles{\Gamma}^{1/2-2\varepsilon}\angles{\Theta}^{\sigma}
\angles{\Sigma_+}^{\sigma-1/2}} \, d\lambda \,
d\eta}_{L^2_{\tau,\xi}}.
\end{align*}

\subsubsection{  Estimate for $I^+_1$ }\label{EstI1+ }
The problem reduces to
$$
H^{1+s-r, \sigma} \cdot H^{s +1/2-2\varepsilon, \sigma} \embeds
  L^2,
$$
which holds by Theorem \ref{prodthm1} for all $1/2<\sigma<1$
provided the conditions
$$s> -1/2 \quad \ r< 1/2+ 2s \quad \text{and} \quad  r\le 1+s$$ are
satisfied, which they are by \eqref{condrs}, and provided also that
$\varepsilon>0$ is sufficiently small, which is tacitly assumed in
the following discussion.

\subsubsection{ Estimate for $I^+_2$ }\label{EstI2+ }
  We assume that $ \abs{\Gamma} \lesssim
\min(\abs{\eta}, \abs{\eta-\xi})= \abs{\eta-\xi}$, since otherwise
$I^{+}$ reduces to $I^{+}_1$ in view of \eqref{rEstimateA}. Giving
up the weight $\angles{\Theta}^{-\sigma+1/2}$ in the integral, we
get
$$
  I^+_2 \lesssim \norm{\int_{\R^{1+3} }
  \frac{ F(\lambda,\eta) G_+(\lambda-\tau,\eta-\xi)}
  {\angles{\eta}^{1+s-r}
\angles{\eta-\xi}^{1/2+s-3\varepsilon}
   \angles{\Sigma_+}^{\sigma} \angles{\Gamma}^{1/2+\varepsilon}} \,
d\lambda \, d\eta}_{L^2_{\tau,\xi}}.
$$
Then the problem reduces to
$$
   H^{1+s-r, 0} \cdot H^{1/2+s-3\varepsilon, \sigma} \embeds
   H^{0, -1/2-\varepsilon}.
$$ But by duality this is equivalent to the embedding
$$
  H^{0, 1/2+\varepsilon} \cdot  H^{1/2+s-3\varepsilon, \sigma} \embeds
     H^{-1-s+r, 0 },
$$
which holds by Theorem \ref{prodthm1} for all $1/2<\sigma<1$
provided
$$
s>0 , \quad   r < 1/2+2s \quad \text{and} \quad  r\le 1+s, $$ which
are true by \eqref{condrs}.

\subsubsection{ Estimate for $I^+_3$ }\label{EstI3+ }
 By the same argument as for $I^+_2$, we assume $\abs{\Gamma}\lesssim
\min(\abs{\eta}, \abs{\eta-\xi})=\abs{\eta-\xi}$. Then
$$
 I^+_3 \lesssim \norm{ \int_{\R^{1+3} }
  \frac{ F(\lambda,\eta) G_+(\lambda-\tau,\eta-\xi)}
  { \angles{\eta}^{1+s-r}
\angles{\eta-\xi}^{1/2+s-3\varepsilon}
   \angles{\Gamma}^{1/2+\varepsilon}\angles{\Theta}^{\sigma}
\angles{\Sigma_+}^{\sigma-1/2}} \, d\lambda \,
d\eta}_{L^2_{\tau,\xi}}.
$$
Hence the problem reduces to proving
\begin{equation}\label{EmbedI+30}
  H^{1+s-r,\sigma} \cdot H^{1/2+s-3\varepsilon, \sigma -1/2} \embeds
H^{0, -1/2-\varepsilon} .
\end{equation}
By duality this is equivalent to the embedding
\begin{equation}\label{EmbedI+3}
   H^{1+s-r,\sigma} \cdot  H^{0, 1/2+\varepsilon} \embeds
H^{-1/2-s+3\varepsilon, -\sigma + 1/2},
\end{equation}
which holds by Theorem \ref{prodthm1} if
 $ s>-1/2$ and $ r < \min( 1/2+2s, 1/2+s)$.  But $s> 0$ by \eqref{condrs},
so \eqref{EmbedI+3} holds for
  $$ r <
1/2+s, \  \text { and \ all } \ \
  1/2<\sigma<1.$$

   If $s > 1$ \ and \ $r \le 1+s$ (see figure \ref{fig:1}), then
\eqref{EmbedI+3} reduces to
$$
H^{0,\sigma} \cdot  H^{0, 1/2+\varepsilon} \embeds
H^{-1/2-s+3\varepsilon, -\sigma + 1/2},
$$ which is true by Theorem \ref{prodthm2} for all
   $1/2<\sigma<1$ .
   If $s > 1/2$ \ and \ $r = 1/2+s$ (this includes $(s,r)\in DF \cup F$,
see figure \ref{fig:1}),  then \eqref{EmbedI+3} becomes
$$
H^{1/2,\sigma} \cdot  H^{0, 1/2+\varepsilon} \embeds
H^{-1/2-s+3\varepsilon, -\sigma + 1/2},
$$ which is true by Theorem \ref{prodthm2} for all
   $1/2<\sigma<1$ .\\

 It remains to prove \eqref{EmbedI+3} for (see figure
\ref{fig:1})  $$ (s, r)\in  D \cup AD \cup BD \cup R_2 \cup R_4.$$
To do this, we need special choices of $\sigma$ which will depend on
$s$ and $r$ as in \eqref{condsigma}. We shall consider five cases
based on these regions. In the rest of the paper, $\theta\in[0,1]$
is an interpolation parameter, $\varrho>0$ depends on $s$ and $r$,
and $\varepsilon, \delta>0$ will be chosen sufficiently small,
depending on $\varrho$. We may also assume that $\varrho
\gg\delta\gg\varepsilon$.
\\
\paragraph{Case 1: $( s, r) \in R_2 $} Then according to \eqref{condsigma} we choose
$\sigma=1/2+s$ (note that $1/2<\sigma<1$, since $0<s<1/2$ in this
region). Write $r=1/2+2s-\varrho$; Then \eqref{EmbedI+3} becomes
\begin{equation}\label{EmbedI+31}
   H^{1/2-s+\varrho,1/2+s} \cdot  H^{0, 1/2+\varepsilon} \embeds
H^{-1/2-s+3\varepsilon, -s} .
\end{equation}
At $s=\delta$, \eqref{EmbedI+31} becomes
  \begin{equation}\label{EmbedI+32}
 H^{1/2-\delta+\varrho,1/2+\delta} \cdot  H^{0, 1/2+\varepsilon}
\embeds  H^{-1/2-\delta+3\varepsilon, -\delta},
  \end{equation} which holds by Theorem \ref{prodthm1}.
   At $s=1/2-\delta$, \eqref{EmbedI+31}
becomes
\begin{equation}\label{EmbedI+33}
 H^{\delta+\varrho,1-\delta} \cdot H^{0,
 1/2+\varepsilon}\embeds   H^{-1+\delta+3\varepsilon, -1/2+\delta}.
  \end{equation} By duality this equivalent to
 $$
  H^{1-\delta-3\varepsilon, 1/2-\delta}\cdot H^{0,
 1/2+\varepsilon}\embeds   H^{-\delta-\varrho,-1+\delta},
 $$   which is true by \eqref{bie1}.
  Now, interpolation
 between  \eqref{EmbedI+32} and \eqref{EmbedI+33} with
$ \theta=\frac{2(s-\delta)}{1-4\delta}$  (note that $0\le \theta \le
 1$ whenever $\delta\le s\le 1/2-\delta$) gives \eqref{EmbedI+31}.
 \\
\paragraph{Case 2: $( s, r) \in  AD$}
 Here $0<s<1/2$, $r=1/2+s$. According to
\eqref{condsigma} we choose
 $\sigma=1/2+s/3$. Then
 \eqref{EmbedI+3} becomes
$$
   H^{1/2,1/2+s/3} \cdot  H^{0, 1/2+\varepsilon} \embeds
H^{-1/2-s+3\varepsilon, -s/3} ,
$$
which holds by \eqref{bie1e} for $s\ge \delta$.\\

\paragraph{Case 3: $ (s,r) \in  R_4 $} By \eqref{condsigma}, we choose
$\sigma=3/2-s+4\varepsilon$ . Since $r\ge1+s$, \eqref{EmbedI+3}
reduces to (using also duality)
$$
 H^{1/2+s-3\varepsilon, 1-s+4\varepsilon}    \cdot  H^{0, 1/2+\varepsilon}
\embeds H^{0,-3/2+s-4\varepsilon} ,
$$ which holds by \eqref{bie2} for $1/2<s\le 1$.\\

\paragraph{Case 4: $ (s,r) \in  BD $}
Here $s=1/2$ and $1< r< 3/2$. According to \eqref{condsigma}, we
choose $\sigma=1-\varepsilon$. Then \eqref{EmbedI+3} after duality
becomes
$$
 H^{1-3\varepsilon, 1/2-\varepsilon}  \cdot  H^{0, 1/2+\varepsilon}
\embeds H^{-3/2+r,-1+\varepsilon}
 $$
which holds by \eqref{bie1}.\\

\paragraph{Case 5: $ (s,r) \in  D \ (\text{i.e}, (s,r)=(1/2, 1) )$}
Then by \eqref{condsigma} we have $\sigma=2/3+\varepsilon$. Hence
\eqref{EmbedI+3} becomes
$$
 H^{1/2, 2/3+\varepsilon}  \cdot  H^{0, 1/2+\varepsilon} \embeds
H^{-1+3\varepsilon,-1/6-\varepsilon},
 $$
which is true by \eqref{bie1e}.

\subsection{Estimate for $I^-$}\label{MinusReduction}

Assume first $\fixedabs{\eta} \ll \fixedabs{\eta-\xi}$. Then
$\abs{\xi} \sim \fixedabs{\eta-\xi}$, so by \eqref{ThetaEstimates},
$$
  \theta_-^2 \sim
\frac{\abs{\xi}\kappa_-}{\fixedabs{\eta}\fixedabs{\eta-\xi}},
$$
 and hence we have the same estimate for $\theta_-$ as for
$\theta_+$. Moreover, by \eqref{rEstimateA} and \eqref{rEstimateB}
we have
 \begin{equation}\label{k-case}
  \kappa_-^{1/2} \lesssim
\angles{\Gamma}^{1/2-2\varepsilon}
\min(\fixedabs{\eta},\fixedabs{\eta-\xi})^{2\varepsilon} +
\angles{\Theta}^{1/2} + \angles{\Sigma_-}^{1/2},
\end{equation}
 so the analysis of
$I^+$ in the previous subsection applies also to $I^-$. The same is
true if $\fixedabs{\eta} \gg \fixedabs{\eta-\xi}$ or $\abs{\xi} \sim
\fixedabs{\eta} \sim \fixedabs{\eta-\xi}$. Hence we assume from now
on that
\begin{equation}\label{FrequencySupport2}
  \abs{\xi} \ll \fixedabs{\eta} \sim \fixedabs{\eta-\xi},
\end{equation}
in $I^-$. By \eqref{FrequencySupport} and \eqref{ThetaEstimates}, we
have
$$
  I^- \lesssim
  \norm{\int_{\R^{1+3}} \frac{\kappa_-^{1/2} F(\lambda,\eta)
G_-(\lambda-\tau,\eta-\xi)}
  { \angles{\xi}^{1-r}\angles{\eta}^{s} \angles{\eta-\xi}^{1/2+s}
\angles{\Gamma}^{1/2-2\varepsilon}
   \angles{\Theta}^{\sigma} \angles{\Sigma_-}^{\sigma}} \, d\lambda \,
d\eta}_{L^2_{\tau,\xi}}.
$$
By \eqref{k-case}, the estimate reduces to
$$
  I_j^- \lesssim \norm{F}_{L^2}
  \norm{G_-}_{L^2}, \ \ j=1,2, 3 ,
$$
where
\begin{align*}
  I^-_1 &= \norm{\int_{\R^{1+3} }
  \frac{ F(\lambda,\eta) G_-(\lambda-\tau,\eta-\xi)}
  {\angles{\xi}^{1-r}
\angles{\eta-\xi}^{1/2+2s-2\varepsilon}
   \angles{\Theta}^{\sigma} \angles{\Sigma_-}^{\sigma}} \, d\lambda \,
d\eta}_{L^2_{\tau,\xi}},\\
  I^-_2 & = \norm{\int_{\R^{1+3} }
  \frac{ F(\lambda,\eta) G_-(\lambda-\tau,\eta-\xi)}
  {  \angles{\xi}^{1-r}\angles{\eta-\xi}^{1/2+2s}
   \angles{\Gamma}^{1/2-2\varepsilon}\angles{\Theta}^{\sigma-1/2}
\angles{\Sigma_-}^{\sigma} } \, d\lambda \, d\eta}_{L^2_{\tau,\xi}},\\
  I^-_3& = \norm{\int_{\R^{1+3} }
  \frac{ F(\lambda,\eta) G_-(\lambda-\tau,\eta-\xi)}
  {\angles{\xi}^{1-r} \angles{\eta}^{1/2+2s}
   \angles{\Gamma}^{1/2-2\varepsilon}\angles{\Theta}^{\sigma}
\angles{\Sigma_-}^{\sigma-1/2}} \, d\lambda \,
d\eta}_{L^2_{\tau,\xi}}.
\end{align*}

By symmetry it suffices to consider $I_1^- $ and $I_2^-$.
\subsubsection{  Estimate for $I^-_1$ }\label{EstI1- } Here the
problem reduces to
$$
  H^{0, \sigma} \cdot  H^{1/2+2s-2\varepsilon, \sigma} \embeds
     H^{-1+r, 0 },
$$
which holds by Theorem \ref{prodthm1} provided
$$
 r \le 1, \quad s>0 , \quad   r< 1/2+2s,
$$ and  $\sigma>1/2$.
Now assuming $r\ge1$, which implies $\angles{\xi}^{r-1}\lesssim
\angles{\eta}^{r-1}+\angles{\eta-\xi}^{r-1}\sim
\angles{\eta-\xi}^{r-1}$, the problem reduces to
$$
  H^{0, \sigma} \cdot  H^{3/2+2s-r-2\varepsilon, \sigma} \embeds
     L^2,
     $$
which is true by Theorem \ref{prodthm1} provided $ r< 1/2+2s $ and
$\sigma>1/2$. Thus, the estimate for $I^-_1$ holds in the desired
region described in figure \ref{fig:1}.

\subsubsection{ Estimate for $I^-_2$ }\label{EstI2- } We may
assume $ \abs{\Gamma} \lesssim \min(\abs{\eta}, \abs{\eta-\xi}) \sim
\abs{\eta-\xi}$, since otherwise $I^{-}$ reduces to $I^{-}_1$.
Giving up the weight $ \angles{\Theta}$, the problem reduces to
$$
  L^2 \cdot H^{1/2+2s-3\varepsilon, \sigma} \embeds
   H^{-1+r, -1/2-\varepsilon}.
$$ By duality this is equivalent to the embedding
$$
  H^{1-r, 1/2+\varepsilon} \cdot  H^{1/2+2s-3\varepsilon, \sigma} \embeds
   L^2  ,
$$ which holds by Theorem \ref{prodthm1} if
$$
r < 1, \quad s>-1/4, \quad   r< 1/2+2s,
$$and  $\sigma>1/2$.
For $r\ge 1$, using the triangle inequality as in the previous
subsection, the problem reduces to
$$
    H^{0, 1/2+\varepsilon} \cdot  H^{3/2+2s-r-3\varepsilon, \sigma} \embeds
    L^2,
$$ which is true by Theorem \ref{prodthm1} if  $ r< 1/2+2s $ and  $\sigma
>1/2$. Thus, the
estimate for $I^-_2$ holds in the desired region described in figure
\ref{fig:1}.\\

\section{Proof of \eqref{Bilinear-DiracD}}\label{proofBilinear-Dirac}
Without loss of generality we take $[ \pm] = + $. Assume
$\psi,\psi'\in \mathcal S(\R^{1+3})$ . In view of the null form
estimate \eqref{NullFormEstimate}, we can reduce
\eqref{Bilinear-Dirac} (write $\rho=1/2+\varepsilon$, as in
\eqref{condrho}) to
\begin{equation}\label{IIplusminus}
  J^{\pm}
  \lesssim \norm{\psi}_{{X_+^{s, \sigma}}}
  \norm{\psi'}_{X_\pm^{-s,1-\sigma-\varepsilon}},
\end{equation}
where now
$$
  J^{\pm} = \norm{\int_{\R^{1+3}} \frac{\theta_\pm}{\angles{\xi}^{r}
\angles{\abs{\tau}-\abs{\xi}}^{1/2+\varepsilon}} \abs{\widetilde
\psi(\lambda,\eta)}
  \abs{\widetilde \psi'(\lambda-\tau,\eta-\xi)} \, d\lambda \,
d\eta}_{L^2_{\tau,\xi}},
$$
and $\theta_\pm = \vangle\bigl(\eta,\pm(\eta-\xi)\bigr)$ as before.
We use the same notation as in the previous section, except that now
$$
  G_\pm(\lambda,\eta) = \angles{\eta}^{-s}
\angles{\lambda\pm\fixedabs{\eta}}^{1-\sigma-\varepsilon}
  \abs{\widetilde \psi'(\lambda,\eta)}.
$$
The low frequency case, $\min(\fixedabs{\eta},\fixedabs{\eta-\xi})
\le 1$ in $J^{\pm}$, follows from a similar argument as in
~\cite{dfs2006}, and hence we do not consider this question here.
From now on we therefore assume that in $J^{\pm}$,
\begin{equation}\label{HighFrequency}
  \fixedabs{\eta}, \fixedabs{\eta-\xi} \ge 1.
\end{equation}

\subsection{Estimate for $J^+$} By \eqref{ThetaEstimates} and
\eqref{HighFrequency},
$$
  J^+ \lesssim
  \norm{\int_{\R^{1+3}} \frac{\kappa_+^{1/2} F(\lambda,\eta)
G_+(\lambda-\tau,\eta-\xi)}
  {\angles{\xi}^{r-1/2} \angles{\eta}^{1/2+s} \angles{\eta-\xi}^{1/2-s}
\angles{\Gamma}^{1/2+\varepsilon}
  \angles{\Theta}^{\sigma} \angles{\Sigma_+}^{1-\sigma-\varepsilon}} \,
d\lambda \, d\eta}_{L^2_{\tau,\xi}}.
$$
By \eqref{rEstimateA} and \eqref{rEstimateB},
$$\kappa_+^{1/2} \lesssim \abs{\Gamma}^{1/2}
 +\abs{\Theta}^{1/2}+
\abs{\Sigma_+}^{1-\sigma-\varepsilon}\fixedabs{\eta-\xi}^{\sigma-1/2+\varepsilon}.$$
Hence the estimate reduces to
$$
  J_j^+ \lesssim \norm{F}_{L^2}
  \norm{G_+}_{L^2}, \ \ j=1,2, 3 ,
$$
where
\begin{align*}
  J^+_1 &= \norm{\int_{\R^{1+3} }
  \frac{ F(\lambda,\eta) G_+(\lambda-\tau,\eta-\xi)}
  {\angles{\xi}^{r-1/2}\angles{\eta}^{1/2+s}
\angles{\eta-\xi}^{1/2-s}
   \angles{\Theta}^{\sigma} \angles{\Sigma_+}^{1-\sigma-\varepsilon}} \,
d\lambda \, d\eta}_{L^2_{\tau,\xi}},\\
  J^+_2 &= \norm{\int_{\R^{1+3} }
  \frac{ F(\lambda,\eta) G_+(\lambda-\tau,\eta-\xi)}
  {  \angles{\xi}^{r-1/2}\angles{\eta}^{1/2+s}
\angles{\eta-\xi}^{1/2-s}\angles{\Gamma}^{1/2+\varepsilon}
    \angles{\Theta}^{\sigma-1/2}\angles{\Sigma_+}^{1-\sigma-\varepsilon}
} \, d\lambda \, d\eta}_{L^2_{\tau,\xi}},\\
  J^+_3& = \norm{\int_{\R^{1+3} }
  \frac{ F(\lambda,\eta) G_+(\lambda-\tau,\eta-\xi)}
  {\angles{\xi}^{r-1/2}\angles{\eta}^{1/2+s}
\angles{\eta-\xi}^{1-s-\sigma-\varepsilon}
   \angles{\Gamma}^{1/2+\varepsilon}\angles{\Theta}^{\sigma} } \,
d\lambda \, d\eta}_{L^2_{\tau,\xi}},
\end{align*}\\
\subsubsection{ Estimate for $J^+_1$ } \label{EstJ1+ }
The problem reduces to
\begin{equation}\label{EmbedJ+1}
   H^{1/2+s, \sigma} \cdot H^{1/2-s, 1-\sigma-\varepsilon}
\embeds
   H^{1/2-r, 0}.
\end{equation}
\\
If $s>1$ and $ r\ge s$, then \eqref{EmbedJ+1} reduces to
$$H^{1/2+s, \sigma} \cdot H^{1/2-s, 1-\sigma-\varepsilon}
\embeds
   H^{1/2-s, 0},$$ which
is true by Theorem \ref{prodthm2}, for all
$1/2<\sigma<1$.\\

It remains to prove \eqref{EmbedJ+1} in the region $R$ (see figure
\ref{fig:1}). We split this into the following five cases:\\

\paragraph{Case 1: $( s, r) \in R_1 $}  Then according to \eqref{condsigma}, we choose
$\sigma=1/2+s/3$. Write $r=1/2+s/3+\varrho$; \eqref{EmbedJ+1}
becomes
$$
   H^{1/2+s, 1/2+s/3} \cdot H^{1/2-s, 1/2-s/3-\varepsilon}
\embeds
   H^{-s/3-\varrho, 0},
$$ which holds by \eqref{bie3} for $0<s<1/2$.
 \\
 \paragraph{Case 2: $( s, r) \in R_2 $}
Then by \eqref{condsigma}, we choose
 $\sigma=1/2+s$.
  Write  $r=1/2+s+\varrho$; \eqref{EmbedJ+1} becomes
$$
   H^{1/2+s, 1/2+s} \cdot H^{1/2-s, 1/2-s-\varepsilon}
\embeds
   H^{-s-\varrho, 0},
$$ which holds by \eqref{bie3} for $0<s< 1/2$.
  \\
\paragraph{Case 3: $( s, r) \in  R_3$} Then according to \eqref{condsigma}, we
choose $\sigma=5/6-s/3+\varepsilon$. Writing $r=1/3+2s/3+\varrho$,
\eqref{EmbedJ+1} becomes
\begin{equation}\label{EmbedJ+16}
   H^{1/2+s, 5/6-s/3+\varepsilon} \cdot H^{1/2-s, 1/6+s/3-2\varepsilon}
\embeds
   H^{1/6-2s/3-\varrho, 0}.
\end{equation}
At $s=1/2$, \eqref{EmbedJ+16} becomes
 \begin{equation}\label{EmbedJ+17}
  H^{1, 2/3+\varepsilon} \cdot H^{0,
1/3-2\varepsilon} \embeds
   H^{-1/6-\varrho, 0}
\end{equation}
which holds by \eqref{bie4}. At $s=1$, \eqref{EmbedJ+16} becomes
\begin{equation}\label{EmbedJ+18}
   H^{3/2, 1/2+\varepsilon} \cdot H^{-1/2,
1/2-2\varepsilon} \embeds
   H^{-1/2-\varrho, 0},
\end{equation}
which is true by Theorem \ref{prodthm2}. Hence we get
\eqref{EmbedJ+16} by interpolating between \eqref{EmbedJ+17} and
\eqref{EmbedJ+18} with $\theta=-1+2s$.
\\

\paragraph{Case 4: $( s, r) \in  BD$}
Here $s=1/2$ and $1< r<3/2$. Then we choose $\sigma=1-\varepsilon$
in view of \eqref{condsigma}. Hence \eqref{EmbedJ+1} becomes
$$
   H^{1, 1-\varepsilon} \cdot L^2
\embeds
   H^{1/2-r, 0},
$$ which holds by Theorem \ref{prodthm2}.\\

\paragraph{Case 5: $( s, r) \in  BD$}
Then in view of \eqref{condsigma}, we choose
$\sigma=3/2-s+4\varepsilon$. Writing $r=1/2+s+\varrho$,
\eqref{EmbedJ+1} reduces to
$$
H^{1/2+s, 3/2-s+4\varepsilon} \cdot H^{1/2-s, -1/2+s-5\varepsilon}
\embeds
   H^{-s-\varrho, 0},
$$ which is true by Theorem \ref{prodthm2} for $1/2<s\le 1$.

\subsubsection{ Estimate for $J^+_2$ }\label{EstJ2+ }
By duality the problem reduces to
\begin{equation}\label{EmbedJ+2}
   H^{-1/2+r,
1/2+\varepsilon} \cdot  H^{1/2-s, 1-\sigma-\varepsilon} \embeds
   H^{-1/2-s, 1/2-\sigma}.
\end{equation}\\
Assume $s>1$ and $ r \ge s$. Then \eqref{EmbedJ+2} reduces to
proving
$$H^{-1/2+s,
1/2+\varepsilon} \cdot  H^{1/2-s, 1-\sigma-\varepsilon} \embeds
   H^{-1/2-s, 1/2-\sigma},$$ which
holds by Theorem \ref{prodthm2} for all
$1/2<\sigma<1$.\\

To prove \eqref{EmbedJ+2} for $(s,r) \in R$, we consider the
following five
cases.\\

\paragraph{Case 1: $( s, r) \in R_1 $}  Then by \eqref{condsigma}, we choose $\sigma=1/2+s/3$.
Writing $r=1/2+s/3+\varrho$, \eqref{EmbedJ+2} becomes
\begin{equation*}\label{EmbedJ+20}
   H^{s/3+\varrho, 1/2+\varepsilon} \cdot H^{1/2-s, 1/2-s/3-\varepsilon}
\embeds
   H^{-1/2-s, -s/3}.
\end{equation*}
At $s=\delta$, this holds by \eqref{bie5}, and at $s=1/2-\delta$
 by \eqref{bie8}; interpolation implies the intermediate cases.\\

\paragraph{Case 2: $( s, r) \in R_2 $}
By \eqref{condsigma}, we choose
 $\sigma=1/2+s$. Then writing  $r=1/2+s+\varrho$ , \eqref{EmbedJ+2} becomes
\begin{equation*}\label{EmbedJ+23}
   H^{s+\varrho, 1/2+\varepsilon} \cdot H^{1/2-s, 1/2-s-\varepsilon}
\embeds
   H^{-1/2-s, -s}.
\end{equation*}
At $s=\delta$, this holds by \eqref{bie5}, and at $s=1/2-\delta$ by
Theorem \ref{prodthm2}; the intermediate cases follows by
interpolation.
\\

 \paragraph{Case 3: $( s, r) \in  R_3 $}
Then according to \eqref{condsigma}, we choose
$\sigma=5/6-s/3+\varepsilon$. Write $r=1/3+2s/3+\varrho$;
\eqref{EmbedJ+2} becomes
\begin{equation}\label{EmbedJ+26}
   H^{-1/6+2s/3+\varrho, 1/2+\varepsilon} \cdot H^{1/2-s,
1/6+s/3-2\varepsilon} \embeds
   H^{-1/2-s,-1/3+s/3-\varepsilon}.
\end{equation}
At $s=1/2$, \eqref{EmbedJ+26} reduces to
\begin{equation}\label{EmbedJ+27}
 H^{1/6+\varrho, 1/2+\varepsilon} \cdot H^{0, 1/3-2\varepsilon}
\embeds
   H^{-1, -1/6}.
\end{equation} Using the triangle inequality $\angles{\eta-\xi}\lesssim
\angles{\xi}+\angles{\eta}$, \eqref{EmbedJ+26} can be reduced to
$$
H^{1/6+\varrho-\delta, 1/2+\varepsilon} \cdot H^{\delta,
1/3-2\varepsilon}\embeds
   H^{-1, -1/6} $$ and $$
 H^{1/6+\varrho, 1/2+\varepsilon} \cdot H^{\delta, 1/3-2\varepsilon} \embeds
   H^{-1+\delta, -1/6},$$
 which both hold by \eqref{bie8}.
 At $s=1$,
\eqref{EmbedJ+26} becomes
\begin{equation}\label{EmbedJ+28}
 H^{1/2+\varrho, 1/2+\varepsilon} \cdot H^{-1/2,
1/2-2\varepsilon} \embeds
   H^{-3/2, -\varepsilon},
\end{equation} which holds by Theorem \ref{prodthm2}.
Interpolation between \eqref{EmbedJ+27} and \eqref{EmbedJ+28} with
$\theta=2s-1$, gives
  \eqref{EmbedJ+26}.
\\

 \paragraph{Case 4: $( s, r) \in BD $}
We choose $\sigma=1-\varepsilon$, by \eqref{condsigma}. Then
\eqref{EmbedJ+2} becomes
$$
   H^{-1/2+r, 1/2+\varepsilon} \cdot L^2
\embeds
   H^{-1, -1/2-\varepsilon},
$$ which is true by Theorem \ref{prodthm2}.\\

 \paragraph{Case 5: $( s, r) \in  R_4 $}
Then by \eqref{condsigma}, we choose $\sigma=3/2-s+4\varepsilon$.
Write $r=1/2+s+\varrho$; \eqref{EmbedJ+2} reduces to
$$
   H^{s+\varrho, 1/2+\varepsilon} \cdot H^{1/2-s, -1/2+s-5\varepsilon}
\embeds
   H^{-1/2-s, -1+s-4\varepsilon},
$$ which holds by Theorem \ref{prodthm2} for $s>1/2$.

\subsubsection{  Estimate for $J^+_3$ }\label{EstJ3+ }
By duality the problem reduces to
\begin{equation}\label{EmbedJ+3}
  H^{1/2+s, \sigma} \cdot  H^{-1/2+r,
1/2+\varepsilon} \embeds
   H^{-1+s+\sigma+\varepsilon, 0}.
\end{equation}

Assume $s>1$ and $ r\ge s$. Then \eqref{EmbedJ+3} reduces to
$$H^{1/2+s, \sigma} \cdot  H^{-1/2+s,
1/2+\varepsilon} \embeds
  H^{-1+s+\sigma+\varepsilon, 0},$$ which
holds by Theorem \ref{prodthm1} for all
$1/2<\sigma<1$.\\

Next, we prove that \eqref{EmbedJ+3} holds for $(s,r) \in R$.\\

\paragraph{Case 1: $( s, r) \in  R_1$}  Then by \eqref{condsigma}, we choose $\sigma=1/2+s/3$.
Write $r=1/2+s/3+\varrho$; \eqref{EmbedJ+3} becomes
$$
   H^{1/2+s, 1/2+s/3} \cdot H^{s/3+\varrho, 1/2+\varepsilon}
\embeds
   H^{-1/2+4s/3+\varepsilon, 0},
 $$
   which is true by Theorem \ref{prodthm1} for $0<s<1/2$.
   \\

 \paragraph{Case 2: $( s, r) \in  R_2$}
   We choose
 $\sigma=1/2+s$, by \eqref{condsigma}.
  Then writing  $r=1/2+s+\varrho$, \eqref{EmbedJ+3} becomes
   $$
   H^{1/2+s, 1/2+s} \cdot H^{s+\varrho, 1/2+\varepsilon}
\embeds
   H^{-1/2+2s+\varepsilon, 0}.
$$ which holds by Theorem
   \ref{prodthm1} for $0<s<1/2$.\\

\paragraph{Case 3: $( s, r) \in  R_3$}.  Then according to \eqref{condsigma}, we choose
$\sigma=5/6-s/3+\varepsilon$. Write $r=1/3+2s/3+\varrho$;
\eqref{EmbedJ+3}
 becomes
$$
   H^{1/2+s, 5/6-s/3+\varepsilon} \cdot H^{-1/6+2s/3+\varrho,
1/2+\varepsilon} \embeds
   H^{-1/6+2s/3 + 2\varepsilon, 0},
$$ which is true by Theorem \ref{prodthm1} for $1/2\le s \le 1$.
\\

\paragraph{Case 4: $( s, r) \in  BD$}.
Here, $s=1/2$ and $1< r<3/2$). By \eqref{condsigma}, we choose
$\sigma=1-\varepsilon$. Then \eqref{EmbedJ+3} becomes
$$
   H^{1, 1-\varepsilon} \cdot H^{-1/2+r, 1/2+\varepsilon}
\embeds
   H^{1/2, 0},
$$ which holds by Theorem \ref{prodthm1}.\\

\paragraph{Case 5: $( s, r) \in  R_4$}
Then by \eqref{condsigma}, we choose $\sigma=3/2-s+4\varepsilon$.
Write $r=1/2+s+\varrho$; \eqref{EmbedJ+3} becomes
$$
   H^{1/2+s, 3/2-s+4\varepsilon} \cdot H^{s+\varrho, 1/2+\varepsilon}
\embeds
   H^{1/2+5\varepsilon, 0},
$$ which is true by Theorem \ref{prodthm1} for $1/2< s\le1$.

\subsection{Estimate for $J^-$}
By the same argument as in subsection \ref{MinusReduction}, we may
assume
$$
\abs{\xi} \ll \fixedabs{\eta} \sim \fixedabs{\eta-\xi}.
$$
Combining this with \eqref{ThetaEstimates} and
\eqref{HighFrequency}, we get
$$
  J^- \lesssim
  \norm{\int_{\R^{1+3}} \frac{\kappa_-^{1/2} F(\lambda,\eta)
G_-(\lambda-\tau,\eta-\xi)}
  {\angles{\xi}^{r} \angles{\eta}^{1/4} \angles{\eta-\xi}^{1/4}
\angles{\Gamma}^{1/2+\varepsilon}
  \angles{\Theta}^{\sigma} \angles{\Sigma_-}^{1-\sigma-\varepsilon}} \,
d\lambda \, d\eta}_{L^2_{\tau,\xi}}.
$$
By  \eqref{rEstimateA} and \eqref{rEstimateB}, we get
$\kappa_-^{1/2} \lesssim \abs{\Gamma}^{1/2}
 +\abs{\Theta}^{1/2}
+
\abs{\Sigma_-}^{1-\sigma-\varepsilon}\fixedabs{\eta-\xi}^{\sigma-1/2+\varepsilon}$.
Hence the estimate reduces to
$$
  J_j^- \lesssim \norm{F}_{L^2}
  \norm{G_-}_{L^2}, \ \ j=1,2, 3 ,
$$
where
\begin{align*}
  J^-_1& = \norm{\int_{\R^{1+3} }
  \frac{ F(\lambda,\eta) G_(\lambda-\tau,\eta-\xi)}
  {\angles{\xi}^{r}\angles{\eta}^{1/2}
   \angles{\Theta}^{\sigma} \angles{\Sigma_-}^{1-\sigma-\varepsilon}} \,
d\lambda \, d\eta}_{L^2_{\tau,\xi}},\\
  J^-_2 &= \norm{\int_{\R^{1+3} }
  \frac{ F(\lambda,\eta) G_-(\lambda-\tau,\eta-\xi)}
  {  \angles{\xi}^{r}\angles{\eta-\xi}^{1/2}
    \angles{\Gamma}^{1/2+\varepsilon}\angles{\Theta}^{\sigma-1/2}\angles{\Sigma_-}^{1-\sigma-\varepsilon}}
\, d\lambda \, d\eta}_{L^2_{\tau,\xi}},\\
  J^-_3 &= \norm{\int_{\R^{1+3} }
  \frac{ F(\lambda,\eta) G_-(\lambda-\tau,\eta-\xi)}
  {\angles{\xi}^{r}\angles{\eta}^{1-\sigma-\varepsilon}
   \angles{\Theta}^{\sigma} \angles{\Gamma}^{1/2+\varepsilon}} \
d\lambda \, d\eta}_{L^2_{\tau,\xi}}.
\end{align*}\\
\subsubsection{ Estimate for $J^-_1$ }\label{EstJ1- } The problem reduces
to the  estimate
\begin{equation}\label{EmbedJ-1}
  H^{1/2, \sigma} \cdot  H^{0,
1-\sigma-\varepsilon} \embeds
   H^{-r, 0}.
\end{equation}

If $s>1$ and $ r\ge s$ , then \eqref{EmbedJ-1} reduces to
$$H^{1/2, \sigma} \cdot  H^{0,
1-\sigma-\varepsilon} \embeds
   H^{-s, 0},$$ which
holds by Theorem \ref{prodthm2} for all
$1/2<\sigma<1$.\\

We now prove \eqref{EmbedJ-1} for $(s,r) \in R$.\\

\paragraph{Case 1: $( s, r) \in   R_1 $}  Then by \eqref{condsigma}, we choose
$\sigma=1/2+s/3$. Write $r=1/2+s/3+\varrho$; \eqref{EmbedJ-1}
becomes
 $$
   H^{1/2, 1/2+s/3} \cdot H^{0, 1/2-s/3-\varepsilon}
\embeds
   H^{-1/2-s/3-\varrho, 0},
$$ which holds by \eqref{bie9} for $0<s<1/2$.
\\

\paragraph{Case 2: $( s, r) \in   R_2 $}
  Then we choose
 $\sigma=1/2+s$, by \eqref{condsigma}. Writing $r=1/2+s+\varrho$,
\eqref{EmbedJ-1} becomes
$$
   H^{1/2, 1/2+s} \cdot H^{0,
1/2-s-\varepsilon} \embeds
   H^{-1/2-s-\varrho, 0},
$$  which is true by \eqref{bie9} for $0<s<1/2$.
\\

\paragraph{Case 3: $( s, r) \in R_3$}  Then according to \eqref{condsigma}, we choose
$\sigma=5/6-s/3+\varepsilon$. Write $r=1/3+2s/3+\varrho$;
\eqref{EmbedJ-1}
 becomes
$$
   H^{1/2, 5/6-s/3+\varepsilon} \cdot H^{0, 1/6+s/3-2\varepsilon}
\embeds
   H^{-1/3-2s/3-\varrho, 0},
$$
which holds by \eqref{bie9} for $1/2\le s\le 1$.\\

\paragraph{Case 4: $( s, r) \in BD$}
Here $s=1/2$ and $1< r<3/2$. We choose $\sigma=1-\varepsilon$ by
\eqref{condsigma}.
 Then
\eqref{EmbedJ-1} becomes
$$ H^{1/2, 1-\varepsilon} \cdot L^2 \embeds
   H^{-r, 0},$$  which is true by Theorem \ref{prodthm2}.\\

   \paragraph{Case 5: $( s, r) \in R_4$}
Then by \eqref{condsigma}, we choose $\sigma=3/2-s+4\varepsilon$.
Write $r=1/2+s+\varrho$; \eqref{EmbedJ-1} becomes
$$
   H^{1/2, 3/2-s+4\varepsilon} \cdot H^{0,
-1/2+s-5\varepsilon} \embeds
   H^{-1/2-s-\varrho, 0},
$$ which is true by Theorem \ref{prodthm2} for $1/2<s\le 1$.
\subsubsection{  Estimate for $J^-_2$ }\label{EstJ2- }
Giving up the weight $\angles{\Theta}^{\sigma-1/2}$ and keep
duality, the problem reduces to
\begin{equation}\label{EmbedJ-2}
  H^{r,
1/2+\varepsilon} \cdot H^{1/2, 1-\sigma-\varepsilon} \embeds L^2.
\end{equation}

  Assume $s>1$ and $ r\ge s$. Then
\eqref{EmbedJ-2} reduces to proving
$$H^{s, 1/2+\varepsilon} \cdot  H^{1/2,
1-\sigma-\varepsilon} \embeds
   L^2,$$ which
holds by Theorem \ref{prodthm2} for all
$1/2<\sigma<1$.\\

It remains to prove \eqref{EmbedJ-2} for $(s,r) \in R$, which we
shall do
in the following five cases.\\

\paragraph{Case 1: $( s, r) \in  R_1 $}  Then by \eqref{condsigma}, we choose $\sigma=1/2+s/3$.
Write $r=1/2+s/3+\varrho$; \eqref{EmbedJ-2} becomes
\begin{equation*}\label{EmbedJ-21}
 H^{1/2+s/3+\varrho, 1/2+\varepsilon} \cdot H^{1/2,
1/2-s/3-\varepsilon} \embeds
  L^2.
\end{equation*}
At $s=\delta$, this holds by \eqref{bie6}, and at $s=1/2-\delta$,
 by \eqref{bie7}; the intermediate cases follows by interpolation.
  \\

\paragraph{Case 2: $( s, r) \in  R_2 $}
We choose
 $\sigma=1/2+s$,  by \eqref{condsigma}. Then writing  $r=1/2+s+\varrho$,
\eqref{EmbedJ-2} becomes
\begin{equation*}\label{EmbedJ-24}
   H^{1/2+s+\varrho, 1/2+\varepsilon} \cdot H^{1/2,
1/2-s-\varepsilon} \embeds
   L^2.
\end{equation*}
At $s=\delta$, this holds by \eqref{bie6}, and at $s=1/2-\delta$
 by Theorem \ref{prodthm2}; interpolation implies the intermediate
 cases.
\\

\paragraph{Case 3: $( s, r) \in  R_3$}  Then according to \eqref{condsigma}, we choose
$\sigma=5/6-s/3+\varepsilon$. Write $r=1/3+2s/3+\varrho$;
\eqref{EmbedJ-2}
 becomes
$$
   H^{1/3+2s/3+\varrho, 1/2+\varepsilon} \cdot H^{1/2,
1/6+s/3-2\varepsilon} \embeds
   L^2,
$$  which holds by \eqref{bie7} for $1/2\le s\le 1$.
\\

\paragraph{Case 4: $( s, r) \in  BD $}
We choose $\sigma=1-\varepsilon$, by \eqref{condsigma}. Then
\eqref{EmbedJ-2} becomes
$$
H^{r, 1/2+\varepsilon} \cdot H^{1/2, 0} \embeds
   L^2,
$$ which is true by Theorem \ref{prodthm2}.\\

\paragraph{Case 5: $( s, r) \in  R_4 $}
Then by \eqref{condsigma}, we choose $\sigma=3/2-s+4\varepsilon$.
 Write $r=1/2+s+\varrho$;
\eqref{EmbedJ-2} becomes
$$
   H^{1/2+s+\varrho, 1/2+\varepsilon} \cdot H^{1/2,
-1/2+s-5\varepsilon} \embeds
   L^2.
$$ which holds by Theorem \ref{prodthm2} for $1/2< s\le 1$.
\subsubsection{  Estimate for $J^-_3$ }\label{EstJ13 } By duality,
 the problem reduces to
\begin{equation}\label{EmbedJ-3}
  H^{r,
1/2+\varepsilon} \cdot H^{1-\sigma-\varepsilon, \sigma} \embeds
   L^{2}.
\end{equation}

If $s>1$ and $ r\ge s$, then \eqref{EmbedJ-3} reduces to
$$  H^{s, 1/2+\varepsilon} \cdot H^{1-\sigma-\varepsilon,
1/2+\varepsilon} \embeds
   L^{2},$$ which holds by Theorem \ref{prodthm1} for all
$1/2<\sigma<1$.
\\

We next prove \eqref{EmbedJ-3} for $(s,r) \in R$.\\
\paragraph{Case 1: $( s, r) \in  R_1 $}  Then by \eqref{condsigma}, we choose $\sigma=1/2+s/3$.
Write $r=1/2+s/3+\varrho$; \eqref{EmbedJ-3} becomes
$$
   H^{1/2+s/3+\varrho, 1/2+\varepsilon} \cdot H^{1/2-s/3-\varepsilon,1/2+s/3} \embeds
  L^2.
$$
 which holds by Theorem \ref{prodthm1} for $0<s<3/2$.
\\

\paragraph{Case 2: $( s, r) \in  R_2 $}
 We choose
 $\sigma=1/2+s$,  by \eqref{condsigma}. Then writing  $r=1/2+s+\varrho$,
\eqref{EmbedJ-3} becomes
$$
   H^{1/2+s+\varrho, 1/2+\varepsilon} \cdot H^{1/2-s-\varepsilon,
1/2+s} \embeds
   L^2,
$$ which holds by Theorem \ref{prodthm1} for $0<s<1/2$.
\\

\paragraph{Case 3: $( s, r) \in R_3$} Then according to \eqref{condsigma}, we choose
$\sigma=5/6-s/3+\varepsilon$. Write $r=1/3+2s/3+\varrho$;
\eqref{EmbedJ-3}
 becomes
$$
   H^{1/3+2s/3+\varrho, 1/2+\varepsilon} \cdot
H^{1/6+s/3-2\varepsilon, 1/2+\varepsilon} \embeds
   L^2.
$$ which holds by Theorem \ref{prodthm1} for $s\ge 1/2$.
\\

\paragraph{Case 4: $( s, r) \in  BD $}
Then by \eqref{condsigma}, we choose $\sigma=1-\varepsilon$. Hence
\eqref{EmbedJ-3} becomes
$$
H^{r, 1/2+\varepsilon} \cdot H^{0, 1/2+\varepsilon} \embeds
   L^2,
$$ which is true by Theorem \ref{prodthm1}.\\

\paragraph{Case 5: $( s, r) \in  R_1 $}
Then by \eqref{condsigma}, we choose $\sigma=3/2-s+4\varepsilon$.
 Write $r=1/2+s+\varrho$; \eqref{EmbedJ-3} becomes
$$
   H^{1/2+s+\varrho, 1/2+\varepsilon} \cdot H^{-1/2+s-5\varepsilon,
1/2+ \varepsilon} \embeds
   L^2.
$$  which holds by Theorem \ref{prodthm1} for $s> 1/2$.

\section{Counterexamples}\label{Optimality}
Here we prove optimality conditions on $s$ and $r$ in Theorem
\ref{Mainthm}, as far as iteration in the spaces $X_\pm^{s,\sigma}$,
$H^{r,\rho}$ is concerned. To be precise, we prove:

\begin{theorem}\label{optmainthm} If \ $s \le 0
$ \ or \ $r \le \frac{1}2$ \ or \ $r<s$ \ or \ $r
> 1 +s$\ or \ $r>\frac{1}2 + 2s$,
  then for all $ \sigma, \rho \in \R$ and $\varepsilon > 0$,
at least one of the estimates \eqref{Bilinear-DiracD} or
\eqref{Bilinear-KG} fails.
\end{theorem}

More generally, we prove:
\begin{theorem}\label{opthmgen}
Let $a_1,a_2,a_3,\alpha_1,\alpha_2,\alpha_3 \in \R$. If the 4-spinor
estimate
$$
   \norm{\innerprod{\beta P_+(D_x)\psi}{P_\pm (D_x)
\psi'}}_{H^{-a_3,-\alpha_3}}
   \lesssim
   \norm{\psi}_{{X_+^{a_1,\alpha_1}}}
   \norm{\psi'}_{X_\pm^{a_2,\alpha_2}},
$$
holds for all $\psi, \psi' \in \mathcal{S}(\R^{1+3})$, then:
\begin{gather}
   \label{cond1}
   a_1+a_2+a_3 \ge \frac{1}2,
   \\
   \label{cond2}
   \frac{a_1+\alpha_1}2+a_2+a_3 \ge \frac34
   \\
   \label{cond3}
   a_1+ \frac{a_2+\alpha_2}2 + a_3 \ge \frac34,
   \\
   \label{cond4}
   a_1+a_3 \ge 0.
   \\
    \label{cond5}
   a_2+a_3 \ge 0.
   \\
    \label{cond6}
   a_1+a_2+\alpha_3 \ge 0.
\end{gather}
\end{theorem}

\subsection{Proof of Theorem \ref{optmainthm}}
Applying \eqref{cond1} and \eqref{cond5} in Theorem \ref{opthmgen}
to \eqref{Bilinear-DiracD}, with $(a_1,a_2,a_3, \alpha_1, \alpha_2,
\alpha_3) = (s,-s,r, \sigma, 1-\sigma-\varepsilon, \rho)$, we see
that the conditions $r\ge 1/2$ and $r\ge s$ are necessary.
Similarly, we apply \eqref{cond1} and \eqref{cond5} in Theorem
\ref{opthmgen} to \eqref{Bilinear-KG}, with
$(a_1,a_2,a_3,\alpha_1,\alpha_2,\alpha_3)
=(s,s,1-r,\sigma,\sigma,1-\rho-\varepsilon)$, to obtain the
necessary conditions $r\le 1/2 +2s$ and $r\le 1+ s$. We further
apply the summation of \eqref{cond2} and \eqref{cond3} to
\eqref{Bilinear-DiracD} to obtain the necessary condition $r>1/2$
($r\ge 1/2+\varepsilon/4$), which is stronger than $r\ge1/2$.
Finally, we combine the necessary conditions $r>1/2$ and $r\le 1/2
+2s$ to conclude that $s>0$ is also a necessary condition.

\subsection{Proof of Theorem \ref{opthmgen}}

The following counterexamples are directly adapted from those for
the 2d case in ~\cite{dfs2006}, and depend on a large, positive
parameter $L $ going to infinity. We choose  $A, B, C \subset \R^3$,
depending on $L$ and concentrated along the $\xi_1$-direction, with
the property
\begin{equation}\label{ABCproperty}
   \eta \in A, \,\, \xi \in C \implies \eta-\xi \in B.
\end{equation}
 Using these sets, we then construct $\psi$ and $\psi'$ depending
on $L$, such that
\begin{equation}\label{LowerBound}
  \frac{\norm{\innerprod{\beta P_{+}(D_x)
\psi}{P_{\pm}(D_x)\psi'}}_{H^{-a_3,-\alpha_3}}}
  {\norm{\psi}_{X_+^{a_1,\alpha_1}}\norm{\psi'}_{X_\pm^{a_2,\alpha_2}}}
  \gtrsim \frac{1}{L^\delta},
\end{equation} for some $\delta=\delta(a_1,a_2,a_3,\alpha_1,\alpha_2, \alpha_3)$. This inequality will lead to the necessary
condition $\delta\ge 0$.

 Let us take
the plus sign in \eqref{LowerBound} for the moment. Later,  we will
also use the minus sign. Assuming $A,B,C$ have been chosen, we set
\begin{align}
  \label{Psi1}
  \widetilde \psi(\lambda,\eta) &= \mathbf{1}_{\lambda+\eta_1 = O(1)}
\mathbf{1}_{\eta \in A} v_+(\eta),
  \\\label{Psi2}
  \widetilde \psi'(\lambda-\tau,\eta-\xi) &=
\mathbf{1}_{\lambda-\tau+\eta_1-\xi_1 = O(1)} \mathbf{1}_{\eta -\xi
\in B} v_+(\eta-\xi),
\end{align}
where  \begin{equation}\label{vPlus}
 v_+(\xi) = \bigl[ 1,  0,
\hat\xi_3,
 \hat\xi_1+i\hat\xi_2\bigr]^T
\end{equation}
is an eigenvector of $P_+(\xi)$, and $\hat\xi \equiv
 \frac{\xi}{\abs{\xi}}$. \\

Observe that
\begin{equation}\label{Eigenproduct}
  \innerprod{\beta v_{+}(\eta)}{v_+(\zeta)}
  = 1 - \hat\eta\cdot\hat\zeta + i\hat{\eta'}\wedge\hat{\zeta'},
  \end{equation}  \qquad \qquad where
  $\hat{\eta'}\wedge\hat{\zeta'}=\hat\eta_1\hat\zeta_2 -
  \hat\eta_2\hat\zeta_1 \ \ \text{and} \ \ \xi'=(\xi_1,\xi_2).$
Hence
\begin{equation}\label{Imagpart}
  \im\innerprod{\beta v_{+}(\eta)}{v_+(\eta-\xi)} = \pm\sin\theta_+ \sim
\pm\theta_+,
\end{equation}
where the sign in front of $\sin\theta_+$ depends on the orientation
of $(\eta',\eta'-\xi')$. But the sets $A,B,C$ will be chosen so that
the orientation of the pair $(\eta',\eta'-\xi')$ is fixed; hence we
conclude (see ~\cite{dfs2006}) that
\begin{multline}\label{Product3}
  \norm{\innerprod{\beta P_{+}(D) \psi}{P_+(D)\psi'}}_{H^{-a_3,-\alpha_3}}
\ge K^+,
  \\
  \text{where}
  \quad
  K^+=
    \norm{\int_{\R^{1+3}}
\frac{\theta_+}{\angles{\xi}^{a_3}\angles{\abs{\tau}-\abs{\xi}}^{\alpha_3}}
  \mathbf{1}_{\left\{
  \scriptstyle\eta \in A
  \atop
  \scriptstyle\lambda+\eta_1 = O(1)
  \right\}
  }
  \mathbf{1}_{\left\{
  \scriptstyle\xi \in C
  \atop
  \scriptstyle\tau+\xi_1 = O(1)
  \right\}
  }
  \, d\lambda \, d\eta}_{L^2_{\tau,\xi}}.
\end{multline}

We now construct the counterexamples, by choosing the sets $A,B,C$.
Note that in $K^+$,
\begin{equation}\label{Slabs}
\begin{gathered}
  \eta \in A, \qquad \xi \in C, \qquad \eta-\xi \in B,
  \\
  \lambda + \eta_1 = O(1),
  \qquad \tau + \xi_1 = O(1),
  \qquad \lambda-\tau + \eta_1-\xi_1 = O(1).
\end{gathered}
\end{equation}
\subsubsection{Necessity of \eqref{cond1}} We consider
high-high frequency interaction giving out put at high frequency.
Set
\begin{align*}
  A &= \left\{ \xi \in \R^3 : \abs{\xi_1 - L} \le L/4, \,\, \abs{\xi_2 -
L^{1/2}} \le L^{1/2}/4, \ \ \abs{\xi_3 - L^{1/2}} \le L^{1/2}/4
\right\},
  \\
  B &= \left\{ \xi \in \R^3 : \abs{\xi_1 - 2L} \le L/2, \,\, \abs{\xi_2}
\le L^{1/2}/2, \ \ \abs{\xi_3} \le L^{1/2}/2 \right\},
  \\
  C &= \left\{ \xi \in \R^3 : \abs{\xi_1 + L} \le L/4, \,\, \abs{\xi_2 -
L^{1/2}} \le L^{1/2}/4, \ \ \abs{\xi_2 - L^{1/2}} \le
L^{1/2}/4\right\}.
\end{align*}

Then \eqref{ABCproperty} holds. By \eqref{Slabs}, we have
$$
  \theta_+ = \vangle(\eta',\eta'-\xi') \sim \frac{1}{L^{1/2}},
  \qquad \abs{\xi}, \fixedabs{\eta}, \fixedabs{\eta-\xi} \sim L,
$$
and
\begin{equation}\label{NearCone1}
  \lambda+\fixedabs{\eta} = \lambda+\eta_1+\fixedabs{\eta}-\eta_1
  = \lambda+\eta_1 + \frac{\eta_2^2 + \eta_3^2}{\fixedabs{\eta}+\eta_1} =
O(1).
\end{equation}
Similarly,
\begin{equation}\label{NearCone2}
  \lambda-\tau+\fixedabs{\eta-\xi} = O(1),
  \qquad
  \bigabs{\abs{\tau}-\abs{\xi}} = \bigabs{\tau-\abs{\xi} }\le \tau+\xi_1=
O(1).
\end{equation}
Let $\abs{A}$ denotes the volume of $A$. Then
$$
  K^+ \sim \frac{\abs{A}\abs{C}^{1/2}}{L^{1/2+a_3}}
  \qquad \text{and} \qquad
  \norm{\psi}_{X_+^{a_1,\alpha_1}}
  \sim L^{a_1} \abs{A}^{1/2}, \quad  \norm{\psi'}_{X_+^{a_2,\alpha_2}}\sim
L^{a_2} \abs{B}^{1/2}
$$
Since $\abs{A} = \abs{C} \sim L^{2}$, we conclude that
\eqref{LowerBound} holds with
$\delta(a_1,a_2,a_3,\alpha_1,\alpha_2,\alpha_3) = a_1+a_2+a_3 -1/2$,
proving the necessity of $a_1+a_2+a_3 \ge 1/2$.

\subsubsection{Necessity of \eqref{cond2} and \eqref{cond3}}\label{HLH1}
We consider high-low frequency interaction with output at high
frequency.
\begin{align*}
  A &= \left\{ \xi \in \R^3 : \abs{\xi_1} \le L^{1/2}/2, \,\, \abs{\xi_2 -
1} \le L^{1/2}/2, \ \abs{\xi_3 - 1} \le L^{1/2}/2 \right\},
  \\
  B &= \left\{ \xi \in \R^3 : \abs{\xi_1 - L} \le L^{1/2}, \,\,
\abs{\xi_2} \le L^{1/2}, \ \abs{\xi_3} \le L^{1/2} \right\},
  \\
  C &= \left\{ \xi \in \R^3 : \abs{\xi_1 + L} \le L^{1/2}/2, \,\,
\abs{\xi_2 - 1} \le L^{1/2}/2, \ \abs{\xi_3 - 1} \le L^{1/2}/2
\right\}.
\end{align*}

Then $\theta_+ =\vangle(\eta',\eta'-\xi') \sim 1$, $\fixedabs{\eta}
\sim L^{1/2}$ and $\abs{\xi}, \fixedabs{\eta-\xi} \sim L$. Further,
\eqref{NearCone2} still holds, whereas the calculation in
\eqref{NearCone1} shows that $\lambda+\fixedabs{\eta} \sim L^{1/2}$,
since $\fixedabs{\eta}+\eta_1 \ge \eta_2 - \eta_1 \ge L^{1/2}/2$.
Thus,
$$
  K^+ \sim \frac{\abs{A}\abs{C}^{1/2}}{L^{a_3}},
  \qquad
  \norm{\psi}_{X_+^{a_1,\alpha_1}} \sim L^{{a_1}/2+{\alpha_1} /2}
\abs{A}^{1/2},
  \qquad \norm{\psi'}_{X_+^{a_2,\alpha_2}}
  \sim L^s \abs{B}^{1/2}.
$$
\\
But $\abs{A}, \abs{B}, \abs{C} \sim L^{3/2}$, hence
\eqref{LowerBound} holds with $\delta(a_1,a_2,a_3,\alpha_1,
\alpha_2, \alpha_3) =\frac{a_1+\alpha_1}2 + a_2+a_3-3/4$, proving
the necessity of \eqref{cond2}.
\\ \\
To show the necessity of \eqref{cond3}, we only need to modify $ A$
and $ B $ such that in $A$, we set $\abs{\xi_1 + L} \le L^{1/2}/2$
instead of $\abs{\xi_1} \le L^{1/2}/2 $, and in $B$ we set $
\abs{\xi_1 } \le L^{1/2}$ instead of $\abs{\xi_1 - L} \le
L^{1/2}/2$. Otherwise, the same argument as above shows the
necessity of \eqref{cond3}.

\subsubsection{Necessity of \eqref{cond4} and \eqref{cond5}}\label{HLH2}
The configuration is the same as in the previous subsection, except
that the squares $A,B,C$ now have side length $\sim 1$. We set
\begin{align*}
  A &= \left\{ \xi \in \R^3 : \abs{\xi_1} \le 1/2, \,\, \abs{\xi_2 - 1}
\le 1/2, \ \abs{\xi_3 - 1} \le 1/2 \right\},
  \\
  B &= \left\{ \xi \in \R^3 : \abs{\xi_1 - L} \le 1, \,\, \abs{\xi_2} \le
1, \ \abs{\xi_3 - 1} \le 1/2 \right\},
  \\
  C &= \left\{ \xi \in \R^3 : \abs{\xi_1 + L} \le 1/2, \,\, \abs{\xi_2 -
1} \le 1/2, \ \abs{\xi_3 - 1} \le 1/2 \right\}.
\end{align*}\\
Then $\theta_+ \sim 1$, $\fixedabs{\eta} \sim 1$, $\abs{\xi},
\fixedabs{\eta-\xi} \sim L$, and \eqref{NearCone1} holds. Since
\eqref{NearCone2} also holds, we conclude:
$$
  K^+ \sim \frac{\abs{A}\abs{C}^{1/2}}{L^{c}},
  \qquad
  \norm{\psi}_{X_+^{a,\alpha}} \sim \abs{A}^{1/2},
  \qquad \norm{\psi'}_{X_+^{b,\beta}}
  \sim L^b \abs{B}^{1/2}.
$$
But $\abs{A}, \abs{B}, \abs{C} \sim 1$, so \eqref{LowerBound} holds
with $\delta(a_1,a_2,a_3,\alpha_1,\alpha_2, \alpha_3) = a_1+a_2$,
proving necessity of \eqref{cond5}. By symmetry \eqref{cond4} is
also necessary.

\subsection{Necessity of \eqref{cond6}.} Here we consider high-high
frequency interaction with output at low frequency, and we choose
the minus sign in \eqref{LowerBound}.
\begin{align*}
  A &= \left\{ \xi \in \R^3 : \abs{\xi_1 - L} \le 1/4, \,\, \abs{\xi_2 -
1} \le 1/4, \ \abs{\xi_3 - 1} \le 1/4  \right\},
  \\
  B &= \left\{ \xi \in \R^3 : \abs{\xi_1 - L} \le 1/2, \,\, \abs{\xi_2}
\le 1/2, \ \abs{\xi_3} \le 1/2 \right\},
  \\
  C &= \left\{ \xi \in \R^3 : \abs{\xi_1} \le 1/4, \,\, \abs{\xi_2 - 1}
\le 1/4, \ \abs{\xi_3} \le 1/2 \right\}.
\end{align*}

We now restrict the integration to
$$
  \eta \in A, \qquad \lambda + \fixedabs{\eta} = O(1),
  \qquad
  \xi \in C, \qquad \tau + 2L = O(1),
$$
which implies
$$
  \eta-\xi \in B, \qquad \lambda-\tau - \fixedabs{\eta-\xi} =
\lambda+\fixedabs{\eta}-\tau-2L+L-\fixedabs{\eta}+L-\fixedabs{\eta-\xi}=
O(1),
$$
since $L-\fixedabs{\eta} = L-\eta_1-(\eta_2^2 +
\eta_3^2)/(\fixedabs{\eta}+\eta_1) = O(1)$ and, similarly,
$L-\fixedabs{\eta-\xi} = O(1)$. Now set
\begin{align*}
  \widetilde \psi(\lambda,\eta) &= \mathbf{1}_{\lambda+\fixedabs{\eta} =
O(1)} \mathbf{1}_{\eta \in A} v_+(\eta),
  \\
  \widetilde \psi'(\lambda-\tau,\eta-\xi) &=
\mathbf{1}_{\lambda-\tau-\fixedabs{\eta-\xi} = O(1)}
\mathbf{1}_{\eta -\xi \in B} v_-(\eta-\xi),
\end{align*}
where $v_-(\xi) = v_+(-\xi)$ and $v_+(\xi)$ is given by
\eqref{vPlus}. Thus, $v_-(\xi)$ is an eigenvector of $P_-(\xi) =
P_+(-\xi)$. Since $
  \theta_- = \vangle(\eta',\xi'-\eta') \sim 1,
$ we then get, arguing as in \eqref{Product3}, and using
\eqref{Eigenproduct},
\begin{multline*}
  \norm{\innerprod{\beta P_{+}(D) \psi}{P_-(D)\psi'}}_{H^{-a_3,-\alpha_3}}
\ge K^-,
  \\
  \text{where}
  \quad
  K^-=
    \norm{\int_{\R^{1+3}}
\frac{1}{\angles{\xi}^{a_3}\angles{\abs{\tau}-\abs{\xi}}^{\alpha_3}}
  \mathbf{1}_{\left\{
  \scriptstyle\eta \in A
  \atop
  \scriptstyle\lambda+\fixedabs{\eta} = O(1)
  \right\}
  }
  \mathbf{1}_{\left\{
  \scriptstyle\xi \in C
  \atop
  \scriptstyle\tau+2L = O(1)
  \right\}
  }
  \, d\lambda \, d\eta}_{L^2_{\tau,\xi}}.
\end{multline*}
Since $\abs{\xi} \sim 1$, $\fixedabs{\eta}, \fixedabs{\eta-\xi} \sim
L$ and $\abs{\tau}-\abs{\xi} \sim \abs{\tau} \sim L$, we see that
$$
  K^- \sim \frac{\abs{A} \abs{C}^{1/2}}{L^{\alpha_3}},
  \qquad
  \norm{\psi}_{X_+^{a_1,\alpha_1}} \sim L^{a_1} \abs{A}^{1/2},
  \qquad \norm{\psi'}_{X_-^{b,\beta}}
  \sim L^{a_2} \abs{B}^{1/2}.
$$
But $\abs{A}, \abs{B}, \abs{C} \sim 1$, hence \eqref{LowerBound}
holds with $\delta(a_1,a_2,a_3,\alpha_1, \alpha_2, \alpha_3) =
a_1+a_2+ \alpha_3$, proving necessity of \eqref{cond6} .

Department of Mathematical Sciences, Norwegian University of Science
and Technology, Alfred Getz' vei 1, N-7491 Trondheim, Norway\\
E-mail address: tesfahun@math.ntnu.no

\end{document}